\documentclass[11pt,a4paper]{article}

\usepackage{amsmath,amssymb,latexsym}

\usepackage{amsthm}
\usepackage{bm}
\usepackage[thinlines]{easymat}
\usepackage{overpic} 
\usepackage{color}
\newtheorem{theorem}{Theorem}

\newtheorem{proposition}{Proposition}

\newtheorem{lemma}{Lemma}
\theoremstyle{definition}
\newtheorem{definition}{Definition}
\newtheorem{remark}{Remark}

\newcommand{\supp}{\mathop{\rm supp}}

\newcommand{\field}[1]{\mathbb{#1}}

\newcommand{\R}{\field{R}}
\newcommand{\Z}{\field{Z}}
\newcommand{\N}{\field{N}}
\newcommand{\C}{\field{C}}

\newcommand{\diag}{\mathop{\rm diag}}

\newcommand{\Ai}{\mathop{\rm Ai}}

\def\XXint#1#2#3{{\setbox0=\hbox{$#1{#2#3}{\int}$}
\vcenter{\hbox{$#2#3$}}\kern-.5\wd0}}

\title{Strong asymptotics for the Pollaczek multiple orthogonal polynomials ensembles
}

\author{A. I. Aptekarev, G. L\'opez Lagomasino, and A. Mart\'{\i}nez--Finkelshtein}

\date{\today}

\begin{document}

\vspace{1cm} \maketitle

\begin{abstract}

We study the asymptotic properties of a class of multiple orthogonal
polynomials with respect to a Nikishin system generated by two
measures $(\sigma_1, \sigma_2)$ with unbounded supports
($\mbox{supp}(\sigma_1) \subset \mathbb{R}_+$,
$\mbox{supp}(\sigma_2) \subset \mathbb{R}_-$), and such that the second
measure $\sigma_2$ is discrete.
The weak asymptotics for these polynomials was obtained by Sorokin in \cite{Sorokin:2009fk}.
We use his result and the Riemann--Hilbert analysis to derive
the strong asymptotics of these polynomials and of the reproducing kernel.


\end{abstract}


\section{Introduction}

If we are given  $p$ weight functions $w_1, \dots, w_p : \R \to \R$ with finite moments and a multi-index $\vec{n} = (n_1,\dots, n_p) \in \mathbb{Z}_+^o \setminus \{\vec 0\}$, the polynomials satisfying the orthogonality relations
\begin{equation} \label{eq3.16}
    \int_{-\infty}^{\infty} P_{\vec{n}}(x) x^k w_j(x) dx = 0  \quad \mbox{for } k=0, \ldots, n_j-1,
    \quad j=1, \ldots, p,
    \end{equation}
are known as \textit{multiple orthogonal polynomials} (or MOP) of type II. These
polynomials  appear in a natural way in certain models of random
matrices and non-intersecting paths, fact that was observed first in
\cite{9} for the Hermitian random matrix model with external
source. The general notion of a  \textit{multiple orthogonal
polynomial ensemble} (generalizing in a certain sense the well-known concept of a biorthogonal  ensemble of A.~Borodin \cite{Borodin98}) was
introduced recently in \cite{Kuij} (see also \cite{aptkuij,MR2827849}):
\begin{definition}
A \emph{multiple orthogonal polynomial ensemble} is a probability density
function on $\mathbb R^n$, with $n= |\vec n|=n_1 + \cdots + n_p$, of the form
\begin{equation} \label{eq3.8}
 \mathcal P(x_1, \ldots, x_n) =     \frac{1}{Z_n}  \det\left[ x_j^{i-1} \right]_{i,j=1, \ldots, n}
    \, \det\left[ \varphi_i(x_j) \right]_{i,j=1, \ldots, n}
    \end{equation}
for certain functions $\varphi_1, \ldots, \varphi_n : \mathbb R \to
\mathbb R $ whose linear span is equal to
\begin{equation} \label{eq3.9}
    \mathrm{span} \{ x^k w_j(x) \mid k = 0, \dots, n_j-1, \, j= 1, \dots, p \}.
    \end{equation}
We say that the MOP ensemble \eqref{eq3.8} is generated by the
weight functions $w_1, \ldots, w_p$ and the multi-index $\vec{n} =
(n_1, \ldots, n_p)$.
\end{definition}

Obviously, the necessary condition for the consistency of this definition is that the product in the right hand side of \eqref{eq3.8} has fixed sign for all $(x_1,\dots, x_n) \in \R^n$, and that the normalizing constant (``partition function'', chosen such that the integral of $ \mathcal P$ over $\R^n$ equals $1$) satisfies
\begin{equation} \label{eq3.10}
    Z_n = \int_{\mathbb R^n}  \det\left[ x_j^{i-1} \right]_{i,j=1, \ldots, n}
    \, \det\left[ \varphi_i(x_j) \right]_{i,j=1, \ldots, n} \, dx_1 \cdots dx_n
        \in \mathbb R \setminus \{0\}.
    \end{equation}
This expression can be transformed (see, for example, in
\cite{aptkuij}) into  a block Hankel determinant
\[ Z_n = c_n \det\begin{bmatrix} H_1 &  \cdots & H_p \end{bmatrix} \]
with $p$ rectangular blocks, where
\[ H_j = \begin{bmatrix} \displaystyle \int_{-\infty}^{\infty} x^{i + k -2} w_j(x) dx
        \end{bmatrix}_{i=1, \ldots, n, \ k = 1, \ldots, n_j}\]
is of size $n \times n_j$ and contains the moments of the weight
$w_j$.

Being a multiple orthogonal polynomial ensemble  a determinantal point
process, there is a kernel $K_n$ such that \eqref{eq3.8}
can be written as
\begin{equation} \label{eq3.12}
    \mathcal P(x_1, \ldots, x_n) = \frac{1}{n!} \det\left[ K_n(x_i,x_j) \right]_{i,j=1, \ldots, n}.
    \end{equation}
 In fact,
\begin{equation} \label{eq3.14}
    K_n(x,y) = \sum_{i=1}^{n} \sum_{j=1}^n \left[A_n^{-1} \right]_{j,i} x^{i-1} \varphi_j(y),
    \end{equation}
where $\left[A_n^{-1} \right]_{j,i}$ denotes the ($ji$)th entry of
the inverse of the matrix
\[ A_n = \left[ \int x^{i-1} \varphi_j(x) dx \right]_{i,j=1, \ldots, n}. \]
With this notation, $Z_n = n! \det A_n $. Since $Z_n \neq 0$, we see that the matrix $A_n$ is
invertible, and the kernel  \eqref{eq3.14} is well-defined.

It is well known that in this context, for every $k = 1, \ldots, n$,
\begin{align} \nonumber
    R(x_1, \ldots, x_k) & := \frac{n!}{(n-k)!} \int_{\mathbb R^{n-k}} \mathcal  P(x_1, \ldots, x_n) \, dx_{k+1} \cdots dx_n \\
    \label{eq3.13}
    & = \det \left[K_n(x_i,x_j) \right]_{i,j=1, \ldots, k}.
\end{align}
Also the (monic) multiple orthogonal polynomial  of type II, $P_{\vec{n}}$, defined by \eqref{eq3.16},  exists, is uniquely determined, and has the probabilistic interpretation of being
the ``average characteristic polynomial'' of the ensemble \eqref{eq3.8}:
\begin{equation} \label{eq3.17}
    P_{\vec{n}}(z) = \mathbb E \left[ \prod_{j=1}^n (z-x_j) \right].
    \end{equation}
This conclusion is based on the integral representation
\begin{equation} \label{eq3.18}
    P_{\vec{n}}(z) = \frac{1}{Z_n} \int_{\mathbb R_n}
        \prod_{j=1}^n (z-x_j) \, \prod_{i < j} (x_j-x_i) \, \det \left[ \varphi_i(x_j) \right]_{i,j=1, \ldots, n} \,  dx_1 \cdots dx_n.
        \end{equation}

Along with the MOPs of type II there is a dual notion of multiple orthogonal polynomials of type
I. These are polynomials $\AA_{\vec{n},j}$, $j=1,
\dots, p$, with
\begin{equation} \label{eq3.19}
    \deg \AA_{\vec{n},j} \leq n_j - 1,
    \end{equation}
and such that the linear form
\begin{equation} \label{eq3.20}
    \mathcal R_{\vec{n}}(x) = \sum_{j=1}^p \AA_{\vec{n}, j}(x) w_j(x)
    \end{equation}
satisfies the orthogonality conditions
\begin{equation} \label{eq3.21}
    \int x^k \mathcal R_{\vec{n}}(x) dx = \begin{cases} 0 & \mbox{ for } j=0, \dots, |\vec{n}| -2, \\
    1 & \mbox{ for } k = |\vec{n}| -1. \end{cases}
    \end{equation}
Again, in the situation of a MOP ensemble \eqref{eq3.8} the MOPs of
type I  and the  form \eqref{eq3.20} uniquely exist. In addition,
$\mathcal R_{\vec{n}}$ satisfies
\[ \int_{-\infty}^{\infty} \frac{\mathcal R_{\vec{n}}(x)}{z-x} dx =
        \mathbb E \left[ \prod_{j=1}^n (z-x_j)^{-1} \right], \qquad z \in \mathbb C \setminus \mathbb R, \]
which means that the Cauchy transform of $\mathcal R_{\vec{n}}$ is the
average of the reciprocal of the characteristic polynomial of a
random point set $x_1, \ldots, x_n$ from the ensemble
\eqref{eq3.8}.

We finish the description of the theoretical background of the MOP ensembles by mentioning the Christoffel-Darboux formula for the kernel \eqref{eq3.14},  found first in \cite{9} (see also \cite{DK}):
\begin{equation} \label{eq3.23}
    (x-y) K_n(x,y) = P_{\vec{n}}(x) \mathcal R_{\vec{n}}(y)
    - \sum_{j=1}^p \frac{h_{\vec{n}},j}{h_{\vec{n}-\vec{e}_j,j}} P_{\vec{n}-\vec{e}_j}(x) \mathcal R_{\vec{n}+\vec{e}_j}(y),
    \end{equation}
where
\begin{equation} \label{eq3.22}
    h_{\vec{n},j} = \int P_{\vec{n}}(x) x^{n_j} w_j(x) dx, \qquad j = 1, \ldots, p,
    \end{equation}
and $\vec{e}_j = ( \delta_{i,j})_{i=1, \ldots, p}$ is the vector of
length $p$ with $1$ at the $j$-th position and $0$ otherwise. It is assumed that
all  multi-indices $\vec{n} \pm \vec{e}_j$ for $j=1, \ldots, p$, are normal, so that the
polynomials and linear forms exist, as well as $h_{\vec{n},j}\neq 0$.
For $p=1$  formula \eqref{eq3.23} reduces to the standard
Christoffel-Darboux formula for orthogonal polynomials.

\medskip

In this paper we consider two weight functions ($p=2$) on the positive semi axis $\R_+=[0,+\infty)$,
$$
w_1(x)=   \frac{1}{\sinh \frac{\pi \sqrt{x}}{2}}, \quad w_2(x)=  \frac{1}{\sqrt{x} \cosh \frac{\pi \sqrt{x}}{2}}  = \frac{\tanh \frac{\pi \sqrt{x}}{2}}{\sqrt{x}}w_1(x) .
$$

For $\vec{n} = (n_1,n_2) \in \mathbb{Z}_+^2 \setminus \{\vec 0\}$ the corresponding  multiple orthogonal polynomial $P_{\vec{n}}$ of type II of degree $\leq n_1 +n_2$, not identically equal to zero, satisfies the conditions
$$
\int_0^{+\infty} x^k P_{\vec{n}}(x)  w_j(x) dx=0 , \qquad k =0,\ldots, n_j-1, \quad j=1,2.
$$
Decomposing $\tanh (\pi z/2)/z$ in simple fractions, it is easy to check that
\[ \frac{ \tanh \frac{\pi \sqrt{z}}{2}}{\sqrt{z}} = \frac{4}{\pi} \sum_{k \geq 0} \frac{1}{z+ (2k+1)^2} = \int\frac{d\sigma_2(x)}{z-x},
\]
where
\[ \sigma_2 =  \frac{4}{\pi} \sum_{k \in {\mathbb{Z}}_+} \delta_{-(2k + 1)^2}\,.
\]
In the well established terminology these formulas mean that the absolutely continuous measures $d\mu_j(x)=w_j(x)dx$ on $\R_+$, $j=1, 2$, form a \emph{Nikishin system} $\mathcal{N}(\sigma_1,\sigma_2)$ generated by $\sigma_1 = \mu_1$, supported on $\mathbb{R}_+$, and the discrete measure $\sigma_2$ whose support is contained in $(-\infty,0)$.
By \cite[Theorem 1.3]{FL} it follows that $P_{\vec{n}}$ is uniquely determined up to a constant factor and $\deg P_{\vec{n}} = n_1 +n_2$; in other words, in our problem all indices ${\vec{n}}\in \mathbb{Z}_+^2 \setminus \{\vec 0\}$ are \emph{normal}\footnote{We wish to remark that $\sigma_2$ is not finite and the two cited theorems from \cite{FL} require this assumption. However, it is easy to see that they are true also in our context: all what is really needed from the measure $\sigma_2$ for their proof is the existence of $\int (z-x)^{-1}d\sigma_2(x)$  for all $z$ outside the support of $\sigma_2$, which is clearly the case.}. Using \cite[Theorem 1.2]{FL} we also know that all zeros of $P_{\vec n}$ are simple and lie in $(0,+\infty)$.
  In the sequel we  normalize $P_{\vec{n}}$ to be monic.

Here we are  interested in the re-scaled asymptotic behavior of the diagonal sequence of polynomials $\left(P_{\vec{n}}\right)$, ${\vec n} = (n,n)$, $n \in \mathbb{N}$.
For simplicity, we adopt the notation $P_n = P_{\vec{n}}$, with $\deg P_n = 2n$.
The monic rescaled polynomials
\begin{equation}
\label{def:rescaled}
Q_n(x) = c_n P_n(4 n^2 x)= x^{2n}+\text{lower degree terms}, \quad c_n = (4n^2)^{-2n},
\end{equation}
are characterized by the orthogonality conditions
\begin{equation} \label{eq:1a}
\int_0^{+\infty} x^kQ_{n}(x) w_{j,n} (x) dx=0 , \qquad k =0,\ldots, n-1, \quad j=1,2,
\end{equation}
where
\begin{equation}
\label{defWrescaled}
w_{1,n}(z):= \frac{1}{\sinh \left(\pi n z^{1/2}\right)} , \quad  w_{2,n}(z):= \frac{1}{z^{1/2} \cosh \left(\pi n z^{1/2}\right)}.
\end{equation}

Given a smooth oriented curve on the plane, we use the subindex $+$ (resp., $-$) to denote the left (resp., the right) side of the curve and the boundary values of any function from the corresponding   side induced by the given orientation. In the case of $\R$ we use standard orientation, so the $+$-side is reached from the upper half plane and the $-$-side from the lower one. Also, unless we explicitly say otherwise, we adopt the convention that $z^{1/2}$ denotes the main branch of the square root in $\C\setminus \R_-$, positive on $\R_+$, while $\sqrt{x}=x^{1/2}$, is its restriction to $x\geq 0$. In particular, the $w_{j,n}$
 are holomorphic and non-vanishing in $\C\setminus \R_+$.

\begin{remark}
It might be convenient to point out the following relation with the multiple orthogonality considered in \cite{MR2470930}. There, the orthogonality weights were (after an appropriate rescaling)
\begin{align*}
\mathfrak{w}_{1,n}(x) & = x^{\alpha/2}\exp\left(-\frac{nx}{t(1-t)}\right)I_{\alpha}\left(\frac{2n\sqrt{ax}}{t}\right), \\
\mathfrak{w}_{2,n}(x) & = x^{(\alpha+1)/2}\exp\left(-\frac{nx}{t(1-t)}\right)I_{\alpha+1}\left(\frac{2n\sqrt{ax}}{t}
\right).
\end{align*}
Using the definition of the modified Bessel function and setting here $\alpha = -1/2$ and $2 \sqrt{a}=\pi t$, we get
\begin{align*}
\mathfrak{w}_{1,n}(x) & = \frac{1}{\pi \sqrt{2 n}}\, \exp\left(-\frac{nx}{t(1-t)}\right)  \frac{1}{x\,  w_{2,n}(x)}, \\
\mathfrak{w}_{2,n}(x) & = \frac{1}{\pi \sqrt{2 n}}\,\exp\left(-\frac{nx}{t(1-t)}\right) \frac{1}{  w_{1,n}(x)}.
\end{align*}
In particular, with this choice of the parameters,
$$
\frac{ w_{2,n}(x)}{ w_{1,n}(x)} = x \, \frac{ \mathfrak{w}_{2,n}(x)}{ \mathfrak{w}_{1,n}(x)},
$$
which explains the connections of our analysis in the following sections with that in \cite{MR2470930}.
\end{remark}

The strong asymptotics of the MOP $Q_n$ is described in the following result:
 \begin{theorem}\label{asymptoticsFinal}
Let
$$
\mathcal H(\zeta)=\frac{  \zeta  }{\sqrt{2}} \left( \frac{   1+\zeta   }{ \zeta^2 +\zeta -1 }\right)^{1/2}
$$
denote the holomorphic branch in $\C \setminus (-\infty, (-1+\sqrt{5})/2)$, normalized by $\mathcal H(1)=1$.
Then:
\begin{enumerate}
\item[(i)]  for $z\in \C\setminus [0, p_+]$, with $p_+=\left( \frac{2}{\sqrt{5}-1}\right)^5$, 
\begin{equation} \label{outerForQ1}	
Q_n(z)=  e^{n g_1(z)}  \, \mathcal H(  \zeta_1 (z))    \left(1 +  \mathcal{O} \left( \frac{1}{n (|z|+1)}\right) \right),
\end{equation}
locally uniformly away from the interval $[0, p_+]$.
\item[(ii)] in a small neighborhood of $(0,p_+)$ in the upper half plane,
\begin{equation} \label{outerForQ2}	
Q_n(z)=  e^{n g_1(z) }    \left(  \mathcal H(  \zeta_1 (z)) +  e^{n \psi(z)} \mathcal H(  \zeta_2 (z))  +  \mathcal{O} \left( \frac{1}{n }\right) \right).
\end{equation}
In particular, on compact subsets of  $(0,p_+)$,
\begin{equation} \label{outerForQ3}	
Q_n(x)=     \left( e^{n g_{1+}(x) }  \mathcal H(  \zeta_{1+} (x)) +  e^{n g_{1-}(x) }  \mathcal H(  \zeta_{1-} (x))  +  \mathcal{O} \left( \frac{1}{n }\right) \right).
\end{equation}

\end{enumerate}
Here $\zeta_1$, $\zeta_2$ are the holomorphic branches of the algebraic function $\zeta(z)$ defined by the equation
$$
z = \frac{1+\zeta}{\zeta^2 (1-\zeta)},
$$	
normalized by
\begin{align*}
\zeta_{1} (z) & = 1 +\mathcal{O} \left(\frac{1}{z } \right),  \quad
\zeta_{2} (z)   =\frac{1}{z^{1/2}}+\mathcal{O} \left(\frac{1}{z } \right),  \quad z\to \infty,
\end{align*}
and $g_1(z)=\int \log (z-t) \, d\lambda_1(t)$ is given in terms of the first component $\lambda_1$ of the  unique solution of a vector equilibrium problem described in Proposition~\ref{teo:2} below.
\end{theorem}

\begin{remark}
Taking into account the definition of $g_1$, formula \eqref{outerForQ1} obviously implies the $n$-th root asymptotic result from \cite{Sorokin:2009fk}:
$$
\lim_n \frac{1}{n} \log |Q_n(z)| = - \mathcal P^{\lambda_1}(z),
$$
where $\mathcal P^\lambda$ is the logarithmic potential of $\lambda$ defined in \eqref{defPotentialandEnergy}.
\end{remark}

Regarding the CD kernel $K_n$, introduced in \eqref{eq3.14}, we have 
 \begin{theorem}
\label{asymptoticsKernel} For the rescaled weights $w_{j,n}$ defined in \eqref{defWrescaled},  the limiting mean density of the positions of the particles from the corresponding multiple polynomial ensemble exists and is supported on $[0,p_+]$:
$$
\lim_{n\to \infty} \frac{1}{n} K_n(x,x)= \lambda_1'(x), \quad x\in (0,p_+),
$$ 
where $\lambda_1$ has the same meaning as in Theorem~\ref{asymptoticsFinal}.

Moreover, for $x^*\in (0,p_+)$,
$$
\lim_{n\to \infty} \frac{1}{n\lambda_1'(x^*)} K_n\left( x^* + \frac{x}{n \lambda_1'(x^*)}, x^* + \frac{y}{n \lambda_1'(x^*)}\right)=\frac{\sin \pi(x-y)}{\pi (x-y)},
$$ 
uniformly for $x$ and $y$ on compact subsets of $\R$.
 \end{theorem}
 
The non-linear steepest descent analysis based on the Riemann-Hilbert characterization of MOP (see Section~\ref{secRH}) allows also to obtain the limit formulas for $Q_n$ and $K_n$ close to the endpoints of $[0, p_+]$. We are not writing these formulas explicitly here, but an interested reader can easily assemble them from the expressions appearing in Section~\ref{sec:RH}.

\section{Riemann-Hilbert characterization} \label{secRH}

The starting point for our analysis is the Riemann-Hilbert interpretation of multiple orthogonality \eqref{eq3.16}, valid for the arbitrary  multi-index ${\vec{n}} = (n_1,n_2)$.

Consider the following Riemann-Hilbert problem  (RHP). Given ${\vec{n}} \in \mathbb{Z}_+^2 \setminus\{\vec 0\}$ find a $3 \times 3$ matrix function $\widehat{\bm{Y}}$, analytic in $\mathbb{C} \setminus \mathbb{R}_+$, such that:
\begin{description}
\item{(RH-$\mathbb{Y}$1)} $\widehat{\bm{Y}}$ has continuous boundary values on $\R_+$ related by the jump condition
$$
\widehat{\bm{Y}}_+(x)=\widehat{\bm{Y}}_-(x)\,\begin{pmatrix} 1 &  w_{1} (x)  &  w_{2} (x) \\
0 & 1 & 0 \\
0 & 0 & 1
\end{pmatrix}, \quad x\in (0,+\infty).
$$
\item{(RH-$\mathbb{Y}$2)} $\widehat{\bm{Y}}(z)=\left( \bm I + \mathcal O(z^{-1}) \right) \diag \left(z^{n_1+n_2}, z^{-n_1}, z^{-n_2} \right)$, as $z\to \infty$, $z\in \C\setminus \R_+ $, where $\bm I$ stands for the $3\times 3$ identity matrix.
\item{(RH-$\mathbb{Y}$3)} $\widehat{\bm{Y}}(z)=\mathcal O \left(\begin{array}{c|c|c}1 & |z|^{-1/2} & |z|^{-1/2}  \end{array}\right)$, $z\to 0$, $z\in \C\setminus \R_+ $.
\end{description}

\begin{proposition}\label{lemma:uniqueness}
For each ${\vec{n}} \in \Z^2_+ \setminus \{\vec 0\}$, the problem (RH-$\mathbb{Y}$1)--(RH-$\mathbb{Y}$3) has a unique solution which is given by the matrix 
\[ \bm{P}_{\vec{n}}(z) =
\begin{pmatrix}
P_{\vec{n}}(z) & \frac{1}{2\pi i}\int \frac{P_{\vec{n}}(x) w_1(x) d x}{x-z} & \frac{1}{2\pi i}\int \frac{P_{\vec{n}}(x) w_2(x)d x}{x-z} \\
d_1 P_{{\vec{n}} - {\vec e}_1}(z) & \frac{d_1}{2\pi i}\int \frac{P_{{\vec{n}} - {\vec e}_1}(x) w_1(x) d x}{x-z} & \frac{d_1}{2\pi i}\int\frac{P_{{\vec{n}} - {\vec e}_1}(x) w_2(x) d x}{x-z} \\
{d_2} P_{{\vec{n}} - {\vec e}_2}(z) & \frac{d_2}{2\pi i}\int \frac{P_{{\vec{n}} - {\vec e}_2}(x) w_1(x) d x}{x-z} & \frac{d_2}{2\pi i}\int\frac{P_{{\vec{n}} - {\vec e}_2}(x) w_2(x) d x}{x-z} \\
\end{pmatrix},
\]
where ${\vec e}_1 = (1,0), {\vec e}_2 = (0,1)$, and
\begin{equation}
\label{coefDD}
 d_j^{-1} = d_{{\vec{n}},j}^{-1} = \frac{-1}{2\pi i}\int x^{n_j-1} P_{{\vec{n}}-{\bf e}_j}(x) w_j(x) dx, \qquad j=1,2.
\end{equation}
\end{proposition}
Notice that the normality of the multi-indices ${\vec{n}} = (n_1,n_2)$ guarantees that the integrals in \eqref{coefDD} are non-vanishing.

\begin{proof}
The proof is basically contained in \cite{Assche01}; however,  the measures there are supported on the whole real line, which slightly simplifies the analysis.  We will sketch a proof here for convenience of the reader.

Using the Sokhotski-Plemelj formula and the orthogonality conditions it is easy to verify that $\bm{P}_{\vec{n}}$ satisfies (RH-$\mathbb{Y}$1)--(RH-$\mathbb{Y}$2). The condition  (RH-$\mathbb{Y}$3) is trivially satisfied by the first column, while for the second and third columns it follows from the fact that  $w_j(x) = \mathcal{O}(1/\sqrt{x}), x \to 0+, j=1,2$.

On the other hand, in the usual manner, from  (RH-$\mathbb{Y}$1)--(RH-$\mathbb{Y}$3) it is easy to deduce that the first column of $\widehat{\bm{Y}}$ has to be made up of multiple orthogonal polynomials with respect to the weights $w_j$ with the multi-indices ${\vec{n}}$, ${\vec{n}} - {\vec e}_1$, and ${\vec{n}} - {\vec e}_2$, respectively, and that the second and third columns must be the corresponding second type functions of these polynomials with respect to $w_1$ and $w_2$ normalized appropriately. The constants appearing in the second and third row are needed to guarantee (RH-$\mathbb{Y}$2). Now uniqueness follows since the multi-indices ${\vec{n}}$, ${\vec{n}} - {\vec e}_1$, and ${\vec{n}} - {\vec e}_2$ are normal, thus the corresponding monic multiple orthogonal polynomials are uniquely determined.
\end{proof}

\begin{remark} Using the expression for $\bm{P}_{\vec{n}}$ it is not difficult to verify that $\det{\bm{P}_{\vec{n}}} \equiv 1$, $z \in {\mathbb{C}}$ (fact established for $\mathbb{C}\setminus [0,+\infty)$ and extended by analyticity to the whole plane),  which is the standard tool for proving the uniqueness of $\bm{P}_{\vec{n}}$. With this approach, only the normality of the multi-index $\vec{n}$ is needed; however, the normality of the other two multi-indices is useful in order to give an explicit description of the second and third rows of $\bm{P}_{\vec{n}}$.
\end{remark}

For the rescaled polynomials \eqref{def:rescaled} we can write an analogous RHP using the  connection between $Q_n$ and $P_n$ and Proposition \ref{lemma:uniqueness}: given $n \in \N$, find a $3 \times 3$ function matrix $\bm Y$ analytic in $\mathbb{C} \setminus \mathbb{R}_+$ such that:
\begin{description}
\item{(RH-Y1)} For $x \in (0,+\infty)$ there is the jump condition
$$
\bm Y_+(x)=\bm Y_-(x)\,\begin{pmatrix} 1 & w_{1,n} (x)  & w_{2,n} (x) \\
0 & 1 & 0 \\
0 & 0 & 1
\end{pmatrix},
$$
\item{(RH-Y2)} $\bm Y(z)=\left( \bm I + \mathcal O(z^{-1}) \right) \diag \left(z^{2n}, z^{-n}, z^{-n} \right)$, as $z\to \infty$, $z\in \C\setminus \R_- $,
\item{(RH-Y3)} $\bm Y(z)=\mathcal O \left(\begin{array}{c|c|c}1 & |z|^{-1/2} & |z|^{-1/2}  \end{array}\right)$, $z\to 0$, $z\in \C\setminus \R_+$.
\end{description}

Indeed, using the change of variables $z \to 4n^2 z$, $x \to 4n^2 x$ it is easy to see that the RHP (RH-Y1)--(RH-Y3) and (RH-$\mathbb{Y}$1)--(RH-$\mathbb{Y}$3) (with ${\vec{n}} = (n,n)$) reduce to one another. From Proposition~\ref{lemma:uniqueness} it follows that (RH-Y1)--(RH-Y3) has a unique solution which may be expressed in terms of $\bm{P}_{\vec{n}}, \vec{n} = (n,n)$. In particular, the first row of $\bm Y$ is
\[ \left(Q_{n}(z),\quad  \frac{1}{2\pi i}\int_0^{+\infty} \frac{Q_{n}(x) w_{1,n}(x) d x}{x-z},\quad \frac{1}{2\pi i}\int_0^{+\infty} \frac{Q_{ n}(x) w_{2,n}(x)d x}{x-z} \right).
\]
The proof of Theorem~\ref{asymptoticsKernel} is based on the following characterization of the kernel $K_n$ in terms of the solution $\bm Y$ of (RH-Y1)--(RH-Y3), see \cite{9,DK}:
\begin{equation}
\label{charactKernel}
K_n(x,y)= \frac{1}{2\pi i (x-y)} \begin{pmatrix}
0& w_{1,n}(y) &w_{2n}(y) 
\end{pmatrix} \, \bm Y_+^{-1}(y) \bm Y_+(x)
 \begin{pmatrix}
1 \\ 0 \\ 0
\end{pmatrix}.
\end{equation}

\section{Equilibrium problem and weak asymptotics}

In the asymptotic analysis of multiple orthogonal polynomials with respect to a general Nikishin system $\mathcal{N}(\sigma_1,\sigma_2)$ in which $\sigma_2$ is discrete, the associated model vector equilibrium problem exhibits an external field acting on $\supp(\sigma_1)$ plus a constraint on $\sigma_2$. This situation is encountered, for example, in \cite{MR2470930}, as well as for Pollaczek weights $w_j$ in \cite{Sorokin:2009fk}; see \cite{ALLMF} for the analysis of a general case.

Let $\mu$ be a positive Borel measure with support contained in $\R$ and satisfying
\begin{equation}\label{C1} \int \log(1+|x|^2) d\mu(x) < \infty,
\end{equation}
(or the equivalent condition $\int \log(1+|x|) d\mu(x) < \infty$, as used in \cite{BKMW}). Its potential and logarithmic energy are defined as
\begin{equation}
\label{defPotentialandEnergy}
 \mathcal{P}^{\mu}(x) := \int \log \frac{1}{| x- y|}  d\mu (y), \quad I(\mu) := \int \int \log \frac{1}{| x- y|} d\mu(x) d\mu (y),
\end{equation}
respectively. From \eqref{C1} it follows that $I(\mu) > - \infty$. Let $\mathcal{M}_e$ be the collection of all measures $\mu$ satisfying \eqref{C1} and for which $I(\mu) < \infty$. If, additionally,
$$
|\mu|=\int_\R d\mu(x)=c, \quad c>0,
$$
 we write $\mu \in \mathcal{M}_e(c)$. When $\mu_1,\mu_2 \in \mathcal{M}_e$, their mutual energy is defined as
\[ I(\mu_1,\mu_2) = \int\int \log \frac{1}{| x- y|} d\mu_1(x) d\mu_2 (y),
\]
which is  finite and
\[ I(\mu_1 - \mu_2) = I(\mu_1) + I(\mu_2) - 2I(\mu_1,\mu_2).
\]
Moreover for $\mu_1,\mu_2 \in \mathcal{M}_e(c)$, we have
\begin{equation} \label{desig} I(\mu_1 - \mu_2) \geq 0
\end{equation}
with equality if and only if $\mu_1 = \mu_2$ (see \cite[Theorem 2.5]{CKL}, \cite[Theorem 2.1]{Sim}, and also \cite[Chapter I]{ST} if the measures have bounded support).

Let $\sigma, \supp (\sigma) = \R_-, |\sigma| > 1$, be a positive Borel measure such that for every compact subset $K \subset \R_-$ we have that $\mathcal{P}^{\sigma|_K}$ is continuous on $\C$. As usual, $\sigma|_K$ denotes the restriction of $\sigma$ to $K$. We define
\[ \mathcal{M}(\sigma) = \{\vec \mu=(\mu_1,\mu_2)^T \in \mathcal{M}_e(2) \times \mathcal{M}_e(1):\,  \supp(\mu_1) \subset \R_+, \supp(\mu_2) \subset \R_-, \mu_2 \leq \sigma\},
\]
where $(\cdot)^T$ stands for transpose.
By $\mu_2 \leq \sigma$ we mean that $\sigma - \mu_2$ is a positive measure. Since we have assumed that $\mathcal{P}^{\sigma|_K}$ is continuous on $\C$ for every compact $K$ it readily follows that $\mathcal{P}^{\mu_2}$ is continuous on $\C$.

Let $\varphi$ be a bounded from below continuous function on $\R_+$. Define
\[ A =
\left(
\begin{array}{cc}
2 & -1 \\
-1 & 2
\end{array}
\right),
\qquad
 f = \left( \begin{array}{cc}
\varphi \\
0
\end{array}
\right).
\]
For $ \vec{\mu}=(\mu_1,\mu_2)^T \in \mathcal{M}(\sigma)$ we introduce the vector function
\[ {\mathcal{W}}^{\vec{\mu}}(x) = \int \log \frac{1}{|x-y|} d A\vec{\mu}(y) + f(x) = (\mathcal{W}_1^{\vec{\mu}}(x),\mathcal{W}_2^{\vec{\mu}}(x))^T
\]
and the functional
\begin{equation} \label{func} J_{\varphi}(\vec{\mu}) = \int ({\mathcal{W}}^{\vec{\mu}}(x) + f(x)) \cdot d \vec{\mu}(x) =  \int (\mathcal{W}_1^{\vec{\mu}}(x) + \varphi(x)) d \mu_1(x) + \int \mathcal{W}_2^{\vec{\mu}}(x) d \mu_2(x).
\end{equation}
Set
\[ J_{\varphi} = \inf \{J_{\varphi}(\vec{\mu}): \vec{\mu} \in \mathcal{M}(\sigma)\}.
\]
Considering  measures with compact support, it is easy to show that there exists $\vec{\mu} \in \mathcal{M}(\sigma)$ such that $J_{\varphi}(\vec{\mu}) < \infty$; therefore, $-\infty \leq J_{\varphi} <  +\infty$.
\begin{definition} \label{def:extremal}
A vector measure $\vec{\lambda} \in \mathcal{M}(\sigma)$ is \emph{extremal} if $J_{\varphi}(\vec{\lambda}) = J_{\varphi} > -\infty$.
\end{definition}

In the study of the existence and uniqueness of an extremal measure one can combine the techniques employed in $\cite{BKMW}$ and \cite{HK} (see also \cite{ALLMF}). In \cite[Definition 1.6]{BKMW} growth conditions at infinity are imposed on the vector external field $f$ which we cannot require here (in fact the second component of $f$ is identically zero). The growth condition is used in \cite[Theorem 1.7]{BKMW} to prove the lower semi-continuity of the functional $J_{\varphi}(\cdot)$ and from there deduce the existence of an extremal measure. However, \cite[Theorem 1.8(b)]{BKMW} remains valid assuming that $J_{\varphi} > -\infty$ and that  a minimizer of the functional exists. From \cite{HK} one can use the more relaxed assumption of weak admissibility of the extremal problem (see Assumption 2.1 therein), sufficient to prove the lower semi-continuity of a certain modified functional which we introduce promptly (see \eqref{def:modified} below). For a detailed discussion see \cite[Section 4]{ALLMF}.

\begin{proposition} \label{prop1} Assume that $J_{\varphi} > -\infty$. The following statements are equivalent:
\begin{itemize}
\item[$(A)$] There  exists $\vec{\lambda} \in \mathcal{M}(\sigma)$ such that $J_{\varphi}(\vec{\lambda}) = J_{\varphi}$.
\item[$(B)$] There  exists $\vec{\lambda} \in \mathcal{M}(\sigma)$ such that $\int W^{\vec{\lambda}} \cdot d(\vec{\nu}-\vec{\lambda}) \geq 0$ for all $\vec{\nu} \in \mathcal{M}(\sigma)$.
\item[$(C)$] There  exist  $\vec{\lambda} = ( \lambda_1,  \lambda_2)^T \in \mathcal{M}(\sigma)$ and constants $\gamma_1,\gamma_2$  such that
\begin{itemize}
\item[$(i)$]
\[ \mathcal{W}_1^{\vec{\lambda}}(x) = 2 {\mathcal{P}}^{ \lambda_1}(x) - {\mathcal{P}}^{ \lambda_2}(x) + \varphi(x)
\left\{
\begin{array}{ll}
= \gamma_1, & x \in \supp( \lambda_1), \\
\geq \gamma_1, & x \in {\mathbb{R}}_+,
\end{array}
\right.
\]
\item[$(ii)$]
\[ \mathcal{W}_2^{\vec{\lambda}}(x) = 2 {\mathcal{P}}^{ \lambda_2}(x) - {\mathcal{P}}^{ \lambda_1}(x)
\left\{
\begin{array}{ll}
\leq \gamma_2, & x \in \supp( \lambda_2), \\
\geq  \gamma_2, & x \in \supp(\sigma -  \lambda_2).
\end{array}
\right.
\]
\end{itemize}
\end{itemize}
If either condition is satisfied, they all have the same unique solution and the constants $\gamma_1,\gamma_2$ are unique as well.
\end{proposition}

Some additional properties are contained in (see \cite[Lemma 4.2]{ALLMF})

\begin{lemma} \label{varios} Let $\vec{\lambda}$ be extremal in the sense of Definition~\ref{def:extremal}.
Then, $\mathcal{P}^{\lambda_1},\mathcal{P}^{\lambda_2}$ are continuous in $\C$, $\supp(\lambda_2)$ is connected, and $0 \in \supp(\lambda_2)$. If $x\varphi'(x)$ is an increasing function on ${\R}_+$ then $\supp(\lambda_1)$ is connected. Should  $\varphi$ be increasing on $\R_+$ then $0 \in \supp(\lambda_1)$. Finally, if
\begin{equation}
\label{condvar*}
\lim_{x \to +\infty} (\varphi(x) - 4\log(x) )= +\infty
\end{equation}
then $\supp(\lambda_1)$ is a compact set, $\supp(\lambda_2) = \R_-$, and the $\lambda_1, \lambda_2$ verify \eqref{C1}.
\end{lemma}

Following \cite{HK} we introduce a modified logarithmic energy of a measure $\mu$ as follows
\[   I^*(\mu) := \iint  \log \frac{\sqrt{1+|x|^2}\sqrt{1+|y|^2}}{| x- y|} d\mu(x) d\mu (y).
\]
Analogously, the mutual energy of $\mu,\nu$ is given by
\[ I^*(\mu,\nu) := \iint  \log \frac{\sqrt{1+|x|^2}\sqrt{1+|y|^2}}{| x- y|} d\mu(x) d\nu (y).
\]
The  advantage of this definition comes from the fact that (see (2.10)-(2.11) in \cite{HK}),
\[ \frac{| x- y|}{\sqrt{1+|x|^2}\sqrt{1+|y|^2}} \leq 1, \quad x,y \in \C.
\]
Therefore, the kernel in the previous integrals is uniformly bounded from below.

Let us introduce
\[ \mathcal{M}^*(\sigma) = \{(\mu_1,\mu_2)^t \in \mathcal{M}(2) \times \mathcal{M}(1): \supp(\mu_1) \subset \R_+,\,  \supp(\mu_2) \subset \R_-, \, \mu_2 \leq \sigma\},
\]
where $\mathcal{M}(c)$ denotes the class of all positive Borel measures with total mass $c >0$. Observe that unlike in the definition of $\mathcal{M}(\sigma)$ we neither assume \eqref{C1}  nor the finiteness of the logarithmic energy of the measures. However, if $\mu, \nu$ verify \eqref{C1} then
\[ I^*(\mu,\nu) = I(\mu,\nu) + \frac{1}{2}|\nu| \int \log(1+|x|^2) d\mu(x) + \frac{1}{2}|\mu| \int \log(1+|x|^2) d\nu(x).
\]
Having this in mind, we introduce the following functional on $\mathcal{M}^*(\sigma)$:
\begin{equation}
\label{def:modified}
J_{\varphi}^*(\vec{\mu}) = 2(I^*(\mu_1) - I^*(\mu_1,\mu_2) + I^*(\mu_2)) + \int (2\varphi - 3 \log(1+|x|^2)) d\mu_1,
\end{equation}
 assuming that $\varphi$ satisfies
\begin{equation} \label{condvar}
\liminf_{x \to \infty} \,(\varphi(x) - 3\log x ) > - \infty,
\end{equation}
so that $J_{\varphi}^*(\vec{\mu}) > -\infty$ for all $\vec{\mu} \in \mathcal{M}^*(\sigma)$. Moreover, $\inf \{J_{\varphi}^*(\vec{\mu}): \vec{\mu} \in \mathcal{M}(\sigma)\} > - \infty$. It is understood that
$J_{\varphi}^*(\vec{\mu}) = +\infty$ when  $I^*(\mu_1) = +\infty$ or $I^*(\mu_2) = +\infty$. Assumption \eqref{condvar} ensures the weak admissibility of the extremal problem. Straightforward calculations yield that $J_{\varphi}^*(\vec{\mu}) = J_{\varphi}(\vec{\mu})$, $\vec{\mu} \in \mathcal{M}(\sigma)$.

The following lemma is a direct consequence of  \cite[Corollary 2.7]{HK}:
\begin{lemma} \label{existe} Assume that $\varphi$ verifies \eqref{condvar}, then $J_{\varphi}^*(\cdot)$ is strictly convex on the set where it is finite and admits a unique minimizer. If the components of the minimizer verify \eqref{C1} then it minimizes $J_{\varphi}(\cdot)$ as well.
\end{lemma}

\smallskip

Summarizing we have that if $\varphi$ satisfies \eqref{condvar*} then $J_{\varphi} > -\infty$ and there exists $\vec{\lambda} \in \mathcal{M}(\sigma)$ such that $J_{\varphi}(\vec{\lambda}) = J_{\varphi}$ which allows us to use Proposition \ref{prop1}.

\medskip

Let us return to the polynomials $Q_n$ satisfying \eqref{eq:1a}. Fix $n \in \N$. We have $(w_{1,n}dx,w_{2,n}dx) = \mathcal{N}(w_{1,n}dx,\sigma_{2,n})$, where
\[
\sigma_{2,n} = \frac{4}{\pi}\sum_{k\geq 0} \delta_{-[(2k+1)/(2n)]^2} .
\]
Using the properties of Nikishin systems (see \cite{FL} and \cite{GRS}) it is easy to deduce that there exists a monic polynomial $Q_{n,2}$, $\deg Q_{n,2} = n,$ whose zeros are simple and contained in the convex hull of $\supp(\sigma_{2,n})$,
such that
\begin{equation} \label{eq:3}
\int  x^{\nu} \frac{Q_{n}(x)}{Q_{n,2}(x)}\frac{dx}{\sinh(\pi n \sqrt{x})}   =0 , \qquad \nu =0,\ldots, 2n -1,
\end{equation}
and
\begin{equation} \label{eq:4}
\int  t^{\nu} \frac{Q_{n,2}(t)}{Q_{n}(t)}\int  \frac{Q_{n}^2(x)}{Q_{n,2}(x)} \frac{d x}{(x-t)\sinh(\pi n \sqrt{x})}d \sigma_{2,n}(t)  =0 , \qquad \nu =0,\ldots,  n  -1.
\end{equation}
That is, $Q_n$ and $Q_{n,2}$ satisfy full orthogonality relations with respect to certain varying measures. From \eqref{eq:3}-\eqref{eq:4} one can establish a connection (see \cite[Section 3.3]{ALLMF} between the asymptotic zero distribution of the sequences of polynomials $\{Q_n\}_{n \geq 0}, \{Q_{n,2}\}_{n\geq 0}$ and the solution of a vector equilibrium problem of the type discussed above in which
\[ \varphi(x) = \pi\sqrt{x}, \qquad d\sigma(x) = \frac{dx}{2\sqrt{|x|}}.
\]
Obviously, $\varphi(x)$ and $x\varphi'(x)$ are increasing on $\R_+$ and \eqref{condvar*} takes place. An explicit solution for the corresponding equilibrium problema is given in \cite[Proposition 3.1]{ALLMF}.

In fact, in \cite{Sorokin:2009fk} V.N. Sorokin proved for this very interesting case the following result which we will use.

\begin{proposition} \label{teo:2}
Let
$$
p_- =- \left( \frac{\sqrt{5}-1}{2}\right)^5\approx -0.09, \quad p_+=\left( \frac{2}{\sqrt{5}-1}\right)^5 \approx 11.09.
$$
There exists a unique pair of measures  $\lambda_1$ and $\lambda_2$, which satisfy the following equilibrium conditions:
\begin{itemize}
\item
$\supp (\lambda_1) =[0,p_+]\subset \overline{\R}_+$,  $|\lambda_1|=2$, and $\supp (\lambda_2) =\R_-$, $|\lambda_2|=1$.
\item $\lambda_2$ is absolutely continuous, and
\begin{equation} \label{constraint}	
\lambda_2'(x) \begin{cases}
= 1/(2\sqrt{|x|}), & x\in [p_-,0], \\
< 1/(2\sqrt{|x|}), & x<p_-.
\end{cases}
\end{equation}
In other words, with
\begin{equation} \label{def_measure_sigma}	
d\sigma(x)=\frac{dx}{2\sqrt{|x|}}=\frac{i \, dx}{2 x_+^{1/2}}, \quad x\in \overline{\R}_-,
\end{equation}
the measure $\sigma-\lambda_2$ is non-negative and supported on $(-\infty, p_-]$.
\item With the external field
\begin{equation} \label{external_field}	
\varphi(x)=\pi \sqrt{x}>0 \text{ on } \R_+,
\end{equation}
there exists a unique constant $\omega\in \R$ such that
\begin{align} \label{equilibrium1}	
2 \mathcal{P}^{\lambda_1}(x) - \mathcal{P}^{\lambda_2}(x)+\varphi(x) &
\begin{cases}
= \omega, & x\in [0,p_+], \\
> \omega, & x>p_+;
\end{cases} \\
\label{equilibrium2}	
2 \mathcal{P}^{\lambda_2}(x) - \mathcal{P}^{\lambda_1}(x) &
\begin{cases}
= 0, &   x\leq p_-,\\
< 0, & x\in (p_-,0).
\end{cases}
\end{align}
\end{itemize}
Moreover.
\begin{equation} \label{vaguelimit} \lim_{n\to \infty} \mu_{Q_n} = \lambda_1/2, \qquad \lim_{n\to \infty} \mu_{Q_{n,2}} = \lambda_2.
\end{equation}
\begin{equation} \label{eq:15*}
\lim_{n\to \infty} \left(\int  \frac{|Q_n (x)|^2}{|Q_{n,2} (x)|}\frac{dx}{\sinh(n \pi \sqrt{x})} \right)^{1/n} = e^{-\omega}.
\end{equation}
and
\begin{equation} \label{eq:23*}
\lim_{n\to \infty} \left(\int   \frac{Q_{n,2}^2(t)}{|Q_n(t)|}\int  \frac{Q_n^2(x)}{|Q_{n,2}(x)|} \frac{d x}{|x-t|\sinh( \pi n \sqrt{x})}d \sigma_{2,n}(t)\right)^{1/n}  = e^{-\omega}.
\end{equation}
\end{proposition}

\medskip

Using the pair of equilibrium measures $\lambda_j$ described in Proposition~\ref{teo:2} we define as usual the \emph{$g$-functions}
\begin{equation} \label{def_g}	
g_j(z)=\int \log (z-t) \, d\lambda_j(t), \quad j=1,2.
\end{equation}
In this definition we understand by $\log (z-\cdot)$ its principal branch in $\C\setminus (-\infty, p_+]$. We summarize next some of their properties needed for out steepest descent analysis.

For the sake of brevity we use the notation
\begin{equation}
\label{defUpsilon}
\upsilon =\upsilon(z):=\exp\left(\pi z^{1/2}\right), \quad z\in \C\setminus \R_-,
\end{equation}
so that $|\upsilon|>1$ in $\C\setminus \R_-$, $\upsilon_+\upsilon_-=1$ on $\R_-$, and
$$
w_{1,n}(z)= \frac{2}{\upsilon^n-\upsilon^{-n}} , \quad w_{2,n}(z)= \frac{2}{z^{1/2}  \left(\upsilon^n+\upsilon^{-n}\right)} .
$$
We have also the following straightforward identities:
\begin{equation} \label{relationsW}	
\begin{split}
w_{1,n}(z) \pm z^{1/2} w_{2,n}(z)& =\frac{4 \upsilon^{\pm n}}{\upsilon^{2n}-\upsilon^{-2n}} 
, \\
\frac{1}{w_{1,n}(z)} \pm \frac{1}{z^{1/2} w_{2,n}(z)} & = \pm \upsilon^{\pm n} , \quad z\in \C\setminus \R_-,
\end{split}
\end{equation}
as well as
\begin{equation} \label{bdryforw}	
w_{j,n+}(x)=- w_{j,n-}(x), \quad x<0, \quad j=1, 2.
\end{equation}

\begin{proposition}
\label{prop:summaryG}
The $g$-functions defined in \eqref{def_g} satisfy the following properties:
\begin{enumerate}
\item[(i)]  $\exp \left( g_{1+} +  g_{1-} - g_2 + \omega \right)(x) \begin{cases}
  = \upsilon , & \text{on $[0,p_+]$,} \\
  < \upsilon, & \text{for $x>p_+$.} \end{cases}$ 
  \item[(ii)]  $\exp \left( g_{2+} +  g_{2-} - g_1 \right)(x) \begin{cases}
  = 1 , & \text{for $x<p_-$,} \\
  >1, & \text{on $(p_-,0]$.} \end{cases}$
   \item[(iii)]  $\exp \left( g_{2+} - g_{2-}  \right)(x) = \upsilon_+^2$ on $[p_-,0]$.

  \item[(iv)] For $x\in [0,p_+]$,
\begin{equation} \label{onsupp1}	
g_{1+}(x)-g_{1-}(x)=2\pi i \int_x^{p_+} d\lambda_1(t).
\end{equation}
 \item[(v)] With $z=x+i y$,
$$
\frac{\partial}{\partial y} \Re \left(2 g_2(z)-g_1(z)- 2\pi z^{1/2} \right)\bigg|_{z=x+i0}= 2\pi \left( \lambda_2'(x)-\frac{1}{2\sqrt{|x|}}\right), \quad x<0.
$$
In particular, this derivative is $\leq 0$ on $\R_-$, and $<0$ for $x<p_-$.
\item[(vi)] There is an open sector with its vertex at $p_-$ and containing $(-\infty, p_-)$ where
\begin{equation}
\label{bound1}
\left| e^{2g_2-g_1}(z) \upsilon^{-2}\right|<1.
\end{equation}
Moreover, there exists an $\varepsilon>0$ such that for $|\Im z|<\varepsilon$, $p_-+\varepsilon < \Re z < -\varepsilon$,
\begin{equation}
\label{bound2}
\left| e^{2g_2-g_1}(z) \upsilon^{-2}\right|>1
\end{equation}
(the relations on $\R_-$ hold in the sense of the boundary values of the left hand sides).
\end{enumerate}

\end{proposition}
\begin{proof}
All these identities are direct  consequence of Proposition~\ref{teo:2}. Indeed, $(i)$, $(ii)$ and $(iii)$ follow directly from \eqref{equilibrium1}, \eqref{equilibrium2} and  \eqref{constraint}, respectively. For $(iv)$ we use the definition of $g_1$. In order to prove $(v)$ we use the equilibrium conditions and the Cauchy--Riemann formulas. Inequality \eqref{bound1} is a consequence of $(v)$, while \eqref{bound2} follows from \eqref{equilibrium2}.

\end{proof}

We introduce finally two other auxiliary functions. From the analyticity of the density of $\lambda_1$ it follows that the right hand side in \eqref{onsupp1} can be extended as a multivalued analytic function to a neighborhood $\mathcal U$ of the interval $[0,p_+]$. Hence, we can define the holomorphic branch
\begin{equation} \label{def:mappingPsi2}	
\psi(z)=-2\pi i \int_z^{p_+} d\lambda_1(t), \quad z\in \mathcal U\setminus (-\infty, p_+].
\end{equation}
By \eqref{onsupp1},
\begin{equation} \label{bdryvalues}	
\psi_\pm(x) = \mp \left( g_{1+}(x)-g_{1-}\right)(x), \quad x\in (0,p_+) \quad \text{and} \quad \psi_\pm(0) =\pm 4 \pi i.
\end{equation}
Since
$$
\frac{d}{dx} \psi(x)= 2\pi  i \lambda_1'(x) \in i\R_+ ,
$$
it implies that
\begin{equation}
\label{normalDer1}
 \frac{\partial}{\partial y} \Re \psi(x+i y)\bigg|_{y=+0}= -2\pi   \lambda_1'(x)< 0, \quad x\in (0,p_+).
\end{equation}

Furthermore, function
$$
\exp \left(2 g_{1 }   - g_2 + \omega + \psi \right)(z) / \upsilon
$$
is holomorphic $B_\delta\setminus (0, p_+)$, where $B_\delta= \{ z: |z-p_+|<\delta\}$, with $\delta<p_+/2$. Considering its boundary values on $B_\delta\cap  (0, p_+)$ and using $(i)$ of Proposition~\ref{prop:summaryG} we conclude that
\begin{equation} \label{identityPsi1}	
\exp \left(2 g_{1 }   - g_2 + \omega  \right)(z)=  \upsilon \exp \left(- \psi(z) \right), \quad z \in B_\delta\setminus (0, p_+).
\end{equation}
Observe that this identity has a holomorphic continuation to $\C\setminus (-\infty, p_+]$, so we can actually use it to extend the definition of $\exp(\psi)$ there:
 \begin{equation} \label{identityPsi1New}	
e^{ \psi(z) } = \upsilon \exp \left(-2 g_{1 }   + g_2 - \omega  \right)(z)    , \quad z \in \C\setminus (-\infty, p_+].
\end{equation}
With this definition, and taking into account \eqref{bdryvalues}, we conclude that
\begin{equation}
\label{psiOnNegative}
e^{ \psi_+(x) } = e^{ \psi_-(x) }, \quad x\in (p_-, 0).
\end{equation}

We can apply analogous arguments when defining
\begin{equation} \label{def:mappingPsi1}	
\widehat \psi(z)=-2\pi i \int_z^{p_-} d\left(\sigma- \lambda_2\right)(t)
\end{equation}
in a neighborhood of $p_-$, cut along $(-\infty, p_-]$.

\section{Non-linear steepest descent analysis} \label{sec:RH}

The starting point of the steepest descent asymptotic analysis is the RHP (RH-Y1)--(RH-Y3) for the matrix $\bm Y$.

\subsection{Global lens opening}

Using the notation \eqref{defUpsilon} we define the following matrix-valued functions in $\C\setminus \R_-$:
$$
\bm A_L(z):= \begin{pmatrix}
1 & -   1/(2 z^{1/2})   \\
 z^{1/2}   & 1/2
\end{pmatrix}  ,
\quad
\bm A_R(z) =    \begin{pmatrix}
1 & -\frac{1}{ z^{1/2} \upsilon^n w_1} \\
 z^{1/2}   & \frac{1}{z^{1/2}  \upsilon^n w_2}
\end{pmatrix} .
$$
We have that
\begin{equation} \label{asymptA}	
\bm A_R(z) =\bm A_L(z)\bm B(z),  \quad \text{with} \quad  \bm B(z):=\begin{pmatrix}
1 &   \frac{1}{2 z^{1/2}    \upsilon^{ 2n}}   \\
   & 1
\end{pmatrix}.
\end{equation}
In particular, $\det \bm A_L(z)=\det \bm A_R(z) \equiv 1$, and 
\begin{equation} \label{asymptAbis}	
\bm A_R(z) =\bm A_L(z) \left(\bm I +\mathcal O\left( |\upsilon|^{-2n}\right)\right), \quad
z\to \infty ,  \quad z\in \C\setminus \R_-. 
\end{equation}
Moreover, by \eqref{relationsW},
\begin{equation} \label{transf1}	
(w_1, w_2) \bm A_R(x)= \left(\frac{4 \upsilon^n}{  \upsilon^{2n}-\upsilon^{-2n} }    , 0\right), \quad x>0,
\end{equation}
and for $x<0$,
\begin{equation} \label{transf2}	
\bm A_{L-}^{-1}(x)\bm A_{L+}(x)= \begin{pmatrix}
0 & - 1/(2x^{1/2}_+) \\
  2  x^{ 1/2}_+  & 0
\end{pmatrix}, \quad \bm A_{R-}^{-1}(x)\bm A_{R+}(x)= \begin{pmatrix}
  \upsilon_+^{2n} & 0 \\
 2  x^{1/2}_+   &   \upsilon_+^{-2n}
\end{pmatrix}.
\end{equation}

 Now we open lenses as in the Figure \ref{fig:first_transformation}, and define the new matrix (written block-wise):
\begin{equation} \label{defX1}	
\bm X(z)= \bm Y(z) \left(\begin{MAT}{c.c}
  1 & \bm 0 \\.
 \bm  0 & \bm A_R(z) \\
\end{MAT} \right) \text{ in the domains limited by } \Delta^\pm  \text{ and } (p_-,+\infty),
\end{equation}
and
\begin{equation} \label{defX2}	
\bm X(z)= \bm Y(z) \left(\begin{MAT}{c.c}
  1 & \bm 0 \\.
 \bm  0 & \bm A_L(z)  \\
\end{MAT} \right)  \text{ in the domains limited by } \Delta^\pm  \text{ and } (-\infty, p_-).
\end{equation}
The newly defined matrix $\bm X$ is the unique solution  of the following RHP:
\begin{description}
\item{(RH-X1)} $\bm X = \bm X_n$ is holomorphic in $\C\setminus (\R\cup \Delta^+ \cup \Delta^-)$, has continuous boundary values on all contours, and these satisfy
$$
\bm X_+(z)=\bm X_-(z) \bm J_{\bm X}(z),
$$
with
\begin{align*}
\bm J_{\bm X}(z)& = \begin{pmatrix}
1 & \frac{4 \upsilon^n}{  \upsilon^{2n}-\upsilon^{-2n} }   & 0 \\
& 1 & \\
& & 1
\end{pmatrix} ,	 \quad x>0,
\\
\bm J_{\bm X}(z)& = \begin{pmatrix}
1 &  &   \\
&  \upsilon_+^{2n} &   \\
 & 2  x^{1/2}_+ & \upsilon_+^{-2n}
\end{pmatrix} ,	 \quad x\in (p_-,0),
\\
\bm J_{\bm X}(z)& = \left(\begin{MAT}{c.c}
  1 & \bm 0 \\.
 \bm  0 & \bm B^{\pm 1}(z) \\
\end{MAT} \right),	 \quad z\in \Delta^\pm,
\\
\bm J_{\bm X}(z)& = \begin{pmatrix}
1 &  &   \\
& 0 & -  1/(2x^{1/2}_+ )  \\
 & 2 x^{ 1/2}_+ & 0
\end{pmatrix} ,	 \quad x\in (-\infty, p_-).
\end{align*}
\item{(RH-X2)} $\bm X(z)=\left( \bm I + \mathcal O(z^{-1}) \right) \left(\begin{MAT}{c.c}
  1 & \bm 0 \\.
 \bm  0 & \bm A_L(z) \\
\end{MAT} \right) \diag \left(z^{2n}, z^{-n}, z^{-n} \right)$ as $z\to \infty$, $z\in \C\setminus \R $.

\item{(RH-X3)} $\bm X(z)=\mathcal O \left(\begin{array}{c|c|c}1 & |z|^{-1/2 } &  |z|^{-1/2 }   \end{array}\right)
   $ as $z\to 0$, and $\bm X(z)=\mathcal O \left(1\right)$ as $z\to p_-$.
\end{description}

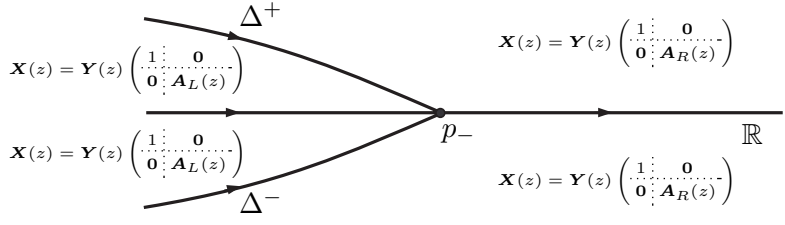
\begin{figure}[t]
\centering \begin{overpic}[scale=1.2]%
{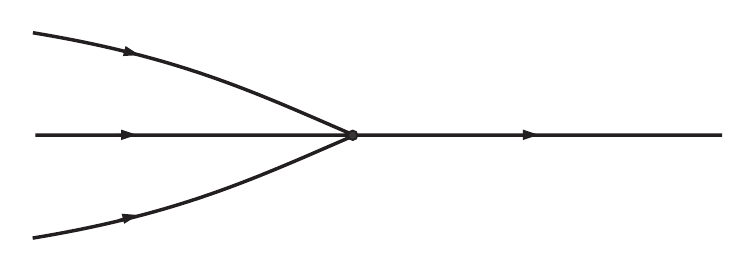}%
      \put(90,14){$\R $}
       \put(18,31){$\Delta^+ $}
       \put(18,4){$\Delta^- $}
       \put(47,15){$p_-$}
 \put(-15,24){\tiny $\bm X(z)= \bm Y(z) \left(\begin{MAT}{c.c}
  1 & \bm 0 \\.
 \bm  0 & \bm A_L(z) \\
\end{MAT} \right)$}
 \put(-15,12){\tiny $\bm X(z)= \bm Y(z) \left(\begin{MAT}{c.c}
  1 & \bm 0 \\.
 \bm  0 & \bm A_L(z) \\
\end{MAT} \right)$}
 \put(55,28){\tiny $\bm X(z)= \bm Y(z) \left(\begin{MAT}{c.c}
  1 & \bm 0 \\.
 \bm  0 & \bm A_R(z) \\
\end{MAT} \right)$}
 \put(55,8){\tiny $\bm X(z)= \bm Y(z) \left(\begin{MAT}{c.c}
  1 & \bm 0 \\.
 \bm  0 & \bm A_R(z) \\
\end{MAT} \right)$}

\end{overpic}
\caption{Global lens opening.}
\label{fig:first_transformation}
\end{figure}

 Indeed, jump relations (RH-X1) are obtained by direct calculations, while (RH-X2) is a consequence of the obvious identity
 $$
\diag \left(z^{2n}, z^{-n}, z^{-n} \right)  \left(\begin{MAT}{c.c}
  1 & \bm 0 \\.
 \bm  0 & \bm C \\
\end{MAT} \right) =\left(\begin{MAT}{c.c}
  1 & \bm 0 \\.
 \bm  0 & \bm C \\
\end{MAT} \right) \diag \left(z^{2n}, z^{-n}, z^{-n} \right),
 $$
valid  for any  $2\times 2$ matrix $\bm C$.

 Finally, (RH-X3) is a result of a direct combination of (RH-Y3) and of the fact that
 $$
 \bm A_R(z)=\mathcal O\begin{pmatrix}
 1 &  1 \\
  |z|^{ 1/2}  &  1
\end{pmatrix}
   \quad \text{as  } z\to 0.
 $$

\subsection{Second transformation}

Now we use the functions $g_j$ defined in \eqref{def_g} in order to normalize the behavior at infinity. Set
\begin{equation}
\label{defUU}
\bm U(z)= \diag\left( e^{n  \omega },   1 ,  1 \right) \bm X(z) \diag\left( e^{-n(g_1(z)+\omega)},   e^{n(g_1(z)-g_2(z))} ,  e^{n g_2(z)}\right) .
\end{equation}
Then $\bm U$ is analytic in $\C\setminus (\R\cup \Delta^+ \cup \Delta^-)$, and
$$
\bm U_+(z)=\bm U_-(z) \bm J_{\bm U}(z),
$$
with
\begin{align*}
\bm J_{\bm U}(z)& = \begin{pmatrix}
e^{-n(g_{1+} -  g_{1-} )(x)}  &   \frac{4 \upsilon^n}{  \upsilon^{2n}-\upsilon^{-2n} }
  e^{n(g_{1+} +  g_{1-} - g_2 + \omega)(x)}& 0 \\
& e^{n(g_{1+} -  g_{1-} )(x)}  & \\
& & 1
\end{pmatrix} ,	 \quad x>0,
\\
\bm J_{\bm U}(z)& = \begin{pmatrix}
1 &  &   \\
&  \upsilon_+^{2n} e^{-n(g_{2+}-g_{2-})} &   \\
 & 2 x^{ 1/2}_+ e^{-n(g_{2+}+g_{2-}-g_1)}  &  \upsilon_+^{-2n} e^{n(g_{2+}-g_{2-})}
\end{pmatrix} ,	 \quad x\in (p_-,0),
\\
\bm J_{\bm U}(z)& = \begin{pmatrix}
1 &  &   \\
& 1 &\pm \frac{1}{2 z^{1/2}    \upsilon^{ 2n}} e^{n(2 g_{2 } -g_1)}  \\
& & 1
\end{pmatrix},	 \quad z\in \Delta^\pm,
\\
\bm J_{\bm U}(z)& = \begin{pmatrix}
1 &  &   \\
& 0 & -  1/(2x^{-1/2}_+)  e^{n(g_{2+}+g_{2-}-g_1)}  \\
 & 2 x^{1/2}_+ e^{-n(g_{2+}+g_{2-}-g_1)} & 0
\end{pmatrix} ,	 \quad x\in (-\infty, p_-).
\end{align*}

Taking into account the properties of the $g$-functions summarized in Proposition~\ref{prop:summaryG} we have that:
\begin{itemize}
\item The   jump matrix $\bm J_{\bm U}$ on $[0,p_+]$ has the form
$$
\bm J_{\bm U}(z)  = \begin{pmatrix}
e^{n \psi_+(x)}  &  \frac{4 }{  1-\upsilon^{-4n} }
 & 0 \\
& e^{n \psi_-(x)}  & \\
& & 1
\end{pmatrix},
$$
while on $(p_+, +\infty)$,
$$
\bm J_{\bm U}(z)  = \begin{pmatrix}
1  &  \frac{4 \upsilon^n}{  \upsilon^{2n}-\upsilon^{-2n} }
  e^{n(g_{1+} +  g_{1-} - g_2 + \omega)(x)}& 0 \\
& 1  & \\
& & 1
\end{pmatrix}=\begin{pmatrix}
1  &  \frac{4 \upsilon^n}{  \upsilon^{2n}-\upsilon^{-2n} }
  e^{n(2 g_{1 }  - g_2 + \omega)(x)}& 0 \\
& 1  & \\
& & 1
\end{pmatrix},
$$
and the entry $(1,2)$ of the jump matrix $\bm J_{\bm U}$ is exponentially decaying for $x>p_+$.
\item The   jump matrix $\bm J_{\bm U}$ on $(p_-,0)$ has the form
$$
\bm J_{\bm U}(z)  = \begin{pmatrix}
1 &  &   \\
&1 &   \\
 & 2x_+^{1/2} e^{-n(g_{2+}+g_{2-}-g_1)}  & 1
\end{pmatrix} ,
$$
and the $(3,2)$ entry is exponentially decaying.
 \item We can choose the contours $\Delta^\pm$ in such a way that the entry $(2,3)$ of the jump matrix $\bm J_{\bm U}$   on $\Delta^\pm$ is also exponentially decaying.
\item The   jump matrix $\bm J_{\bm U}$ for $x<p_-$ has the form
$$
\bm J_{\bm U}(z)  = \begin{pmatrix}
1 &  &   \\
& 0 & -1/(2x_+^{1/2})   \\
 & 2x_+^{1/2}  & 0
\end{pmatrix}.
$$
\end{itemize}
In summary, $\bm J_{\bm U}$ is exponentially close to the identity matrix $\bm I$ on all contours, except on  $\supp(\lambda_1) \cup \supp(\sigma-\lambda_2)$. Furthermore,
$$
\bm U(z)=\left( \bm I + \mathcal O(z^{-1}) \right) \left(\begin{MAT}{c.c}
  1 & \bm 0 \\.
 \bm  0 & \bm A_L(z) \\
\end{MAT} \right)  \quad \text{as} \quad z\to \infty , \quad z\in \C\setminus \R .
$$
Clearly, $\bm U$ has the same behavior at $z=p_-$ and at the origin as $\bm X$, see (RH-X3).

\subsection{Third transformation}

\begin{figure}[t]
\centering \begin{overpic}[scale=1]%
{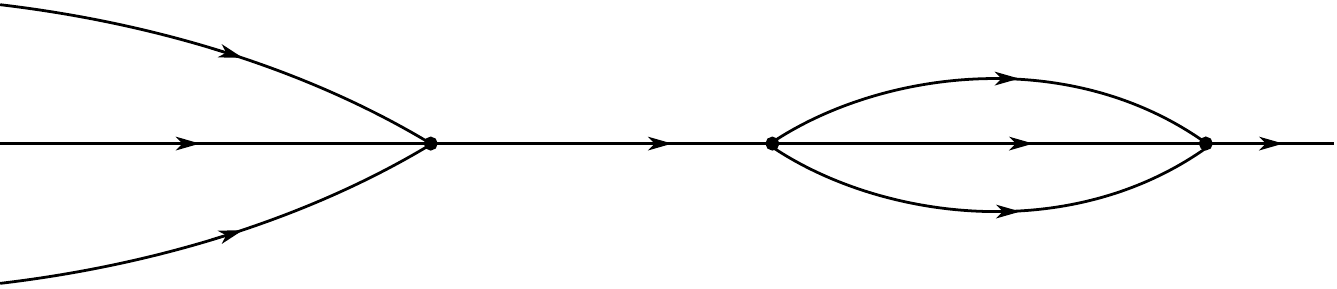}%
      \put(97,12){$\R $}
      \put(90,8.5){$p_+ $}
       \put(20,18){$\Delta^+ $}
       \put(20,2){$\Delta^- $}
       \put(70,16.5){$\Gamma^+ $}
       \put(70,3){$\Gamma^- $}
       \put(32,8.5){$p_-$}
\put(56,8.5){$0$}
\end{overpic}
\caption{Contours for the third transformation.}
\label{fig:third_transformation}
\end{figure}
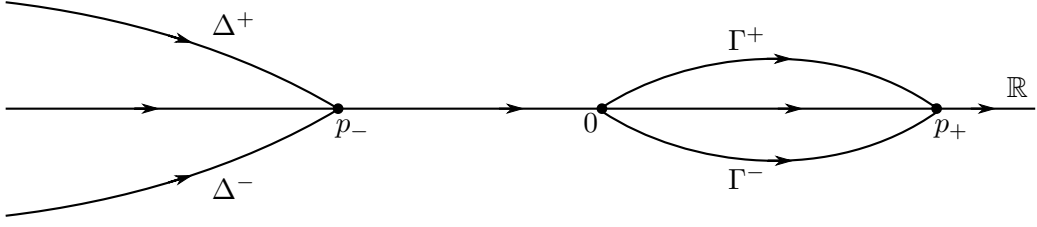

We fix the jump on $[0,p_+]$ observing that by \eqref{bdryvalues},
\begin{equation*}
\begin{split}
\begin{pmatrix}
1   &0 & 0\\
-\frac{1-\upsilon^{-4n}  }{4} e^{n \psi_-(x)}  & 1 & 0 \\
0& 0& 1
\end{pmatrix}  \begin{pmatrix}
e^{n \psi_+(x)}  &  \frac{4}{1-\upsilon^{-4n}  } & 0 \\
0& e^{n \psi_-(x)}   &0 \\
0&0 & 1
\end{pmatrix}\\
\times
\begin{pmatrix}
1   & 0&0 \\
-\frac{1-\upsilon^{-4n}  }{4} e^{n \psi_+(x)}  & 1 & 0 \\
0 & 0& 1
\end{pmatrix}
=\begin{pmatrix}
 0  &  \frac{4}{1-\upsilon^{-4n}  } & 0  \\
- \frac{1-\upsilon^{-4n}  }{4} &  0  & 0\\
0& 0& 1
\end{pmatrix}.
\end{split}
\end{equation*}

So, we open up the new lenses around $[0,p_+]$, as shown in Figure  \ref{fig:third_transformation}, and define the new matrix
\begin{equation} \label{defT}	
\bm T(z)=  \bm U(z) \begin{pmatrix}
1   & 0&0 \\
\mp \frac{1-\upsilon^{-4n}  }{4} e^{n \psi(z)}  & 1 & 0 \\
0 & 0& 1
\end{pmatrix} ,
\end{equation}
for $z$ in the domains bounded by $\Gamma^\pm $ and $[0,p_+]$ (we take ``$-$'' in \eqref{defT} for $\Im z>0$, and ``$+$'' otherwise), and $\bm T(z)=  \bm U(z) $ elsewhere. Hence, $\bm T$ is holomorphic in $\C\setminus (\R\cup \Gamma^\pm \cup \Delta^\pm)$,
and
$$
\bm T_+(z)=\bm T_-(z) \bm J_{\bm T}(z),
$$
with
\begin{align*}
\bm J_{\bm T}(z)& = \begin{pmatrix}
 0  &  \frac{4}{1-\upsilon^{-4n}  } & 0  \\
- \frac{1-\upsilon^{-4n}  }{4} &  0  & 0\\
0& 0& 1
\end{pmatrix},	 \quad x\in (0,p_+),
\\
\bm J_{\bm T}(z)& =
\begin{pmatrix}
1  &  \frac{4}{(1-\upsilon^{-4n})  \upsilon^{n}}
  e^{n(g_{1+} +  g_{1-} - g_2 + \omega)(x)}& 0 \\
& 1  & \\
& & 1
\end{pmatrix}, \quad x>p_+,
\\
\bm J_{\bm T}(z)& = \begin{pmatrix}
1 &  &   \\
&1 &   \\
 & 2 x_+^{1/2} e^{-n(g_{2+}+g_{2-}-g_1)}  & 1
\end{pmatrix} ,	 \quad x\in (p_-,0),
\\
\bm J_{\bm T}(z)& = \begin{pmatrix}
1 &  &   \\
& 1 &\pm \frac{1}{2  z^{1/2} \upsilon^{2n}}e^{n(2 g_{2 } -g_1)}  \\
& & 1
\end{pmatrix},	 \quad z\in \Delta^\pm,
\\
\bm J_{\bm T}(z)& = \begin{pmatrix}
1 &  &   \\
& 0 & - 1/(2x_+^{1/2}  )  \\
 &  2x_+^{1/2}  & 0
\end{pmatrix} ,	 \quad x\in (-\infty, p_-),
\\
\bm J_{\bm T}(z)& = \begin{pmatrix}
1   & 0&0 \\
 \frac{1-\upsilon^{-4n}  }{4} e^{n \psi(z)}  & 1 & 0 \\
& 0& 1
\end{pmatrix} ,	 \quad z\in \Gamma^\pm.
\end{align*}
By \eqref{normalDer1}, we can always choose $\Gamma^\pm$ in such a way that the $(2,1)$ entry of $\bm J_{\bm T}$ is exponentially decaying on $\Gamma^\pm$, away from their endpoints $0$ and $p_+$. Obviously, by definition,
$$
\bm T(z)=\left( \bm I + \mathcal O(z^{-1}) \right) \left(\begin{MAT}{c.c}
  1 & \bm 0 \\.
 \bm  0 & \bm A_L(z) \\
\end{MAT} \right)  \quad \text{as} \quad z\to \infty , \quad z\in \C\setminus \R .
$$

Since
$$
\frac{1-\upsilon^{-4n}  }{4} e^{n \psi(z)} =\mathcal O(|z|^{1/2}), \quad z \to 0,
$$
we get that $\bm T$ has the same asymptotic behavior at $z=p_-$ and at the origin as $\bm X$, see (RH-X3), when $z=0$ is approached both from inside and outside the contours $\Gamma^\pm $.

\subsection{Global parametrix}

Observing the jumps $\bm J_{\bm T}$ above, we can infer that an appropriate model for $\bm T$ is a matrix $\bm N$, holomorphic in $\C\setminus ((-\infty, p_-] \cup [0,p_+])$, such that
\begin{align} \label{RH-N1}
\bm N_+(x)& = \bm N_-(x) \begin{pmatrix}
 0  &  4 & 0  \\
- 1/4 &  0  & 0\\
0& 0& 1
\end{pmatrix},	 \quad x\in (0,p_+),
\\  \label{RH-N2}
\bm N_+(x)& = \bm N_-(x) \begin{pmatrix}
1 & 0 & 0  \\
0& 0 & -1/(2 x_+^{1/2})    \\
0 & 2 x_+^{1/2}  & 0
\end{pmatrix} ,	 \quad x\in (-\infty, p_-),
\end{align}
and
\begin{equation} \label{RH-N3}
\bm N(z)=\left( \bm I + \mathcal O(z^{-1}) \right) \left(\begin{MAT}{c.c}
  1 & \bm 0 \\.
 \bm  0 & \bm A_L(z) \\
\end{MAT} \right)   \quad \text{as} \quad z\to \infty , \quad z\in \C\setminus \R
\end{equation}
(matching asymptotically the behavior of $\bm T$). Observe that this behavior at infinity is consistent with the jump on $(-\infty, p_-)$, see \eqref{transf2}.

This RHP is solved using the Riemann surface $\mathcal R$ constructed gluing the three copies of $\C$, as shown in Figure \ref{fig:Riemann_sufrace}. This is a surface of genus 0.

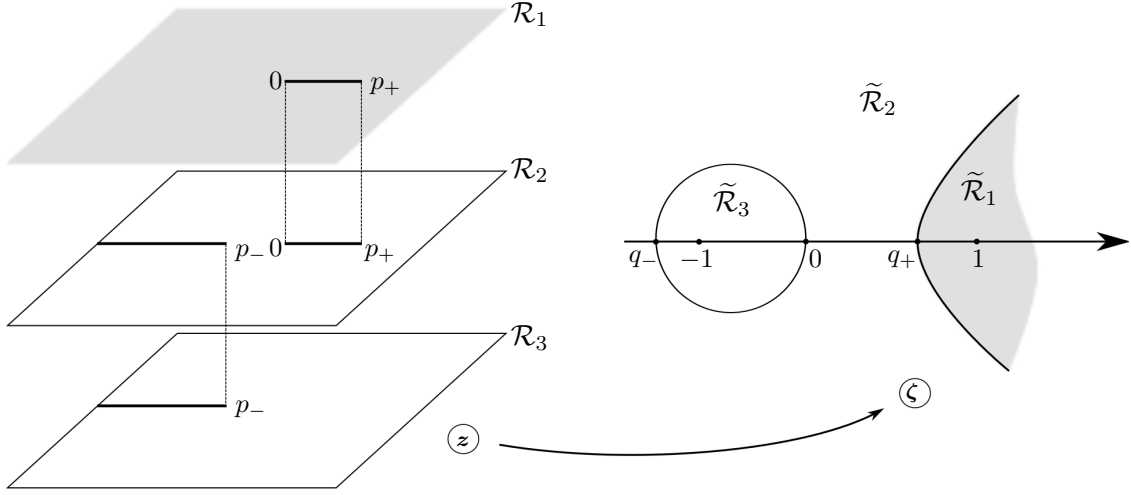
\begin{figure}[h]
\centering \begin{overpic}[scale=0.9]%
{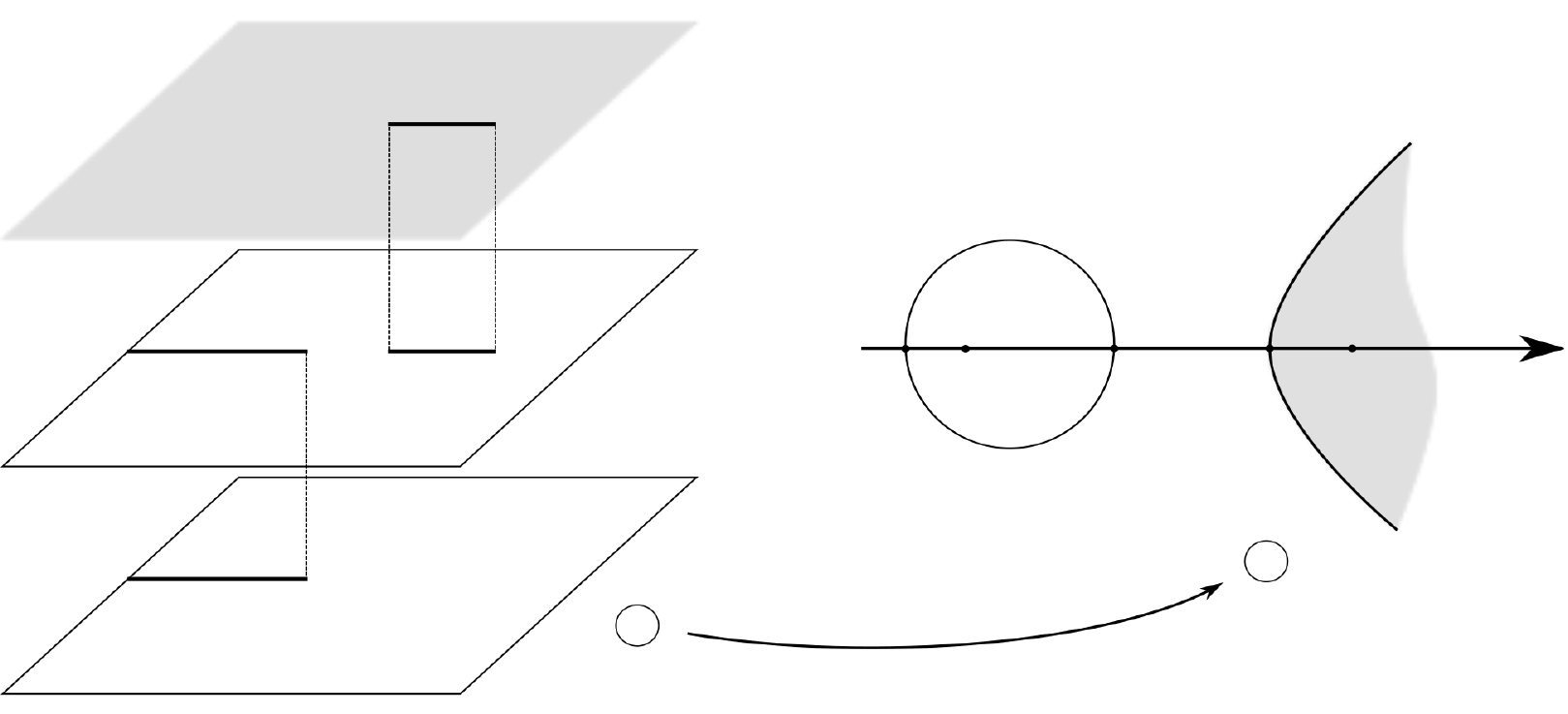}%
\put(44,43){ $ \mathcal R_1 $}
\put(44,29){ $ \mathcal R_2 $}
\put(44,14){ $ \mathcal R_3 $}
\put(31.5,37){ \small $p_+ $}
\put(31,22){ \small $p_+ $}
\put(22.5,37){ \small $0 $}
\put(22.5,22){ \small $0 $}
\put(19.5,22){ \small $p_- $}
\put(19.5,8){ \small $p_- $}
\put(84,27){ $ \widetilde{\mathcal R}_1 $}
\put(75,35){ $ \widetilde{\mathcal R}_2 $}
\put(62,26){ $ \widetilde{\mathcal R}_3 $}
\put(85,21){ \small $1$}
\put(54.5,21.5){ \small $q_-$}
\put(59,21){ \small $-1$}
\put(77.5,21.5){ \small $q_+$}
\put(70.5,21){ \small $0$}
\put(40,5){\scriptsize $\boldsymbol z $}
\put(80,9.3){\scriptsize $\boldsymbol \zeta $}
\end{overpic}
\caption{The Riemann surface $\mathcal R$ and the mapping \eqref{riemann_mapping}.}
\label{fig:Riemann_sufrace}
\end{figure}

It was observed in  \cite{Sorokin:2009fk} that
\begin{equation} \label{riemann_mapping}	
z = \frac{1+\zeta}{\zeta^2 (1-\zeta)}
\end{equation}
establishes a one-to-one bijection between  $\mathcal R$ and the extended $\zeta$-plane\footnote{Formally, we can say that $\mathcal R$ is parametrized by $(z, \zeta)$, where $z$ is the canonical projection on $\C$, and both variables are related by \eqref{riemann_mapping}. Then the bijection is $(z, \zeta) \leftrightarrow \zeta$.}. There are three inverse functions to \eqref{riemann_mapping}, which we choose
such that as $z\to\infty$,
\begin{align}
\zeta_{1} (z) & = 1- \frac{2}{z} -\frac{6}{z^2} 
+\mathcal{O} \left(\frac{1}{z^{3}} \right), \label{zeta1}\\
\zeta_{2} (z) & =\frac{1}{z^{1/2}}+\frac{1}{z} +
\frac{3}{2 z^{3/2}} + \frac{3}{ z^2} + \frac{55}{8 z^{5/2}}+\mathcal{O} \left(\frac{1}{z^{3}} \right), \label{zeta2}\\
\zeta_{3} (z) & = -\frac{1}{z^{1/2}}+\frac{1}{z} -
\frac{3}{2 z^{3/2}} + \frac{3}{ z^2} - \frac{55}{8 z^{5/2}}+\mathcal{O} \left(\frac{1}{z^{3}} \right).\label{zeta3}
\end{align}
Again, all fractional powers are taken
as principal branches, that is, positive on $\R_+$, with the
branch cut along $\R_-$.

Figure \ref{fig:Riemann_sufrace} also shows the domains
\[ \widetilde{\mathcal R}_j = \zeta_j(\mathcal R_j), \qquad j=1,2,3, \]
where $\mathcal R_j$ is the $j$th sheet of the Riemann surface, and
 the location of the points
\begin{equation} \label{qPlusMinus}	
 \zeta_2(0)=\infty,\quad  \zeta_2(\infty)=0, \quad q_\pm:=\zeta_2(p_\pm)=-\frac{1}{2} \pm \frac{\sqrt{5} }{2}
\end{equation}
in the $\zeta$-plane. We observe that
$\zeta_{2+}(-\infty, p_-)$ and $\zeta_{2+}(0,p_+)$ are in the lower half plane,
while $\zeta_{2-}(-\infty, p_-)$ and $\zeta_{2-}(0,p_+)$ are in the upper half plane.

Let us define the following functions:
\begin{equation} \label{def_ff12}	
r_1(\zeta)=\frac{(2\zeta)^{1/2}}{\left( 1+\zeta\right)^{1/2}}, \quad r_2(\zeta)=\frac{1}{4}\, r_1(\zeta), \quad \text{for } \zeta \in \widetilde{\mathcal R}_1 \cup \widetilde{\mathcal R}_2,
\end{equation}
as well as
\begin{equation} \label{def_ff3}
r_3(\zeta)=
\begin{cases}
- \frac{1}{(2\zeta)^{1/2}\left( 1-\zeta\right)^{1/2}}, \quad \zeta \in \widetilde{\mathcal R}_3 \cup (\Im \zeta >0), \\
 \frac{1}{(2\zeta)^{1/2}\left( 1-\zeta\right)^{1/2}}, \quad \zeta \in \widetilde{\mathcal R}_3 \cup (\Im \zeta <0).
\end{cases}
 \quad  \text{for } \zeta \in \widetilde{\mathcal R}_3 .
\end{equation}
In these formulas we use the main branch of all square roots: the branch of $\left( 1+\zeta\right)^{1/2}$ in $\C\setminus (-\infty, -1]$ takes the value $1$ at $\zeta=0$;  the branch of  $\left( 1-\zeta\right)^{1/2}$ is fixed in $\C\setminus [1, +\infty)$ by its value $1$ at $\zeta=0$,  and $\zeta^{1/2}$ is holomorphic in $\C\setminus (-\infty, 0]$ and positive on the positive semi axis.

With this convention, each  function $r_j$ is holomorphic in its corresponding domain $\widetilde{\mathcal R}_j$, $j=1,2,3$.

Now we define
$$
f_j(z)=r_j(\zeta_j(z)), \quad j=1, 2,3;
$$
in a slight abuse of notation, we identify here $z$ on the Riemann surface $\mathcal R$, and its projection on the complex plane.
In this fashion, $f_1$ is holomorphic in $\C\setminus [0,p_+]$, $f_2$ is holomorphic in $\C\setminus ((-\infty, p_-]\cup  [0,p_+])$, and $f_3$ is holomorphic in $\C\setminus (-\infty, p_-]$. Observe also that by \eqref{zeta1}--\eqref{zeta3}, as $z\to\infty$,
\begin{equation} \label{asymptoticsFjatinfinity}	
f_1(z)= 1+ \mathcal{O} \left(\frac{1}{z} \right), \quad f_2(z)= \frac{z^{-1/4}}{2^{3/2}}\left( 1+  \mathcal{O} \left(\frac{1}{z^{1/2}} \right)\right), \quad f_3(z)= \frac{i}{\sqrt{2}} \, z^{1/4}\left( 1+  \mathcal{O} \left(\frac{1}{z^{1/2}} \right)\right).
\end{equation}
For $z \in (0, p_+)$ we have
$$
f_{1\pm}(z)  = r_1(\zeta_{1\pm}(z)) =r_1(\zeta_{2\mp}(z)) = 4 r_2(\zeta_{2\mp}(z)) = 4 f_{2\mp}(z) . 
$$
On the other hand, let $z\in (-\infty, p_-)$; since $\zeta_{2+}(-\infty, p_-)$ is in the lower half plane, we have
\begin{align*}	
f_{3-}(z)  & = r_3(\zeta_{3-}(z)) =r_3(\zeta_{2+}(z)) =   \frac{1}{(2\zeta_{2+}(z))^{1/2}\left( 1-\zeta_{2+}(z)\right)^{1/2}} \\
& =   2 r_2(\zeta_{2+}(z)) \frac{\left( 1+\zeta_{2+}(z)\right)^{1/2}}{ \zeta_{2+}(z)\left( 1-\zeta_{2+}(z)\right)^{1/2}} .
\end{align*}	
By  \eqref{riemann_mapping} and \eqref{zeta2},  for $z\in \C\setminus (-\infty,p_+]$,
$$
\frac{\left( 1+\zeta \right)^{1/2}}{\zeta \left( 1-\zeta \right)^{1/2}} =z^{1/2}, \quad \zeta=\zeta_{2}(z).
$$
Hence,
$$
f_{3-}(z) =  2  (z^{1/2})_+ f_{2+}(z) , \quad z \in (-\infty, p_-).
$$
Analogously,
\begin{align*}	
f_{3+}(z)  & = r_3(\zeta_{3+}(z)) =r_3(\zeta_{2-}(z)) =  -\frac{1}{(2\zeta_{2-}(z))^{1/2}\left( 1-\zeta_{2-}(z)\right)^{1/2}} \\
& =  -2 r_2(\zeta_{2-}(z)) \frac{\left( 1+\zeta_{2-}(z)\right)^{1/2}}{ \zeta_{2-}(z)\left( 1-\zeta_{2-}(z)\right)^{1/2}} = 2  (z^{1/2})_+ f_{2-}(z) , \quad z \in (-\infty, p_-).
\end{align*}	

Gathering all these formulas we conclude that if we define
$$
\widehat{\bm N}(z) = \bm N(z)\, \diag(f_1(z), f_2(z), f_3(z)) ,\quad z \in \C\setminus ((-\infty, p_-]\cup  [0,p_+]),
$$
then we will obtain the following  RH problem for $\widehat{\bm N}$: it is holomorphic in $\C\setminus ((-\infty, p_-] \cup [0,p_+])$,
\begin{align}
\label{Ntilde1}
\widehat{\bm N}_+(z)& = \widehat{\bm N}_-(z) \begin{pmatrix}
 0  &  1 & 0  \\
- 1  &  0  & 0\\
0& 0& 1
\end{pmatrix},	 \quad z\in (0,p_+),
\\
\label{Ntilde2}
\widehat{\bm N}_+(z)& = \widehat{\bm N}_-(z) \begin{pmatrix}
1 & 0 & 0  \\
0& 0 & -1  \\
0 & 1  & 0
\end{pmatrix} ,	 \quad z\in (-\infty, p_-),
\end{align}
and
\begin{align*}
\widehat{\bm N}(z) & = 2^{-3/2}
\left( \bm I + \mathcal O\left(\frac{1}{z}\right) \right)  \diag(1 , z^{-1/4}, z^{1/4}) \begin{pmatrix}
1 & 0 & 0  \\
0& 1 & -i   \\
0 &1 & i
\end{pmatrix} \\
&\times \diag(1+ \mathcal{O}(z^{-1}),1+ \mathcal{O}(z^{-1/2}),1+ \mathcal{O}( z^{-1/2}))  \\
& = 2^{-3/2} \left( \bm I + \mathcal O(z^{-1/2}) \right) \diag(1 , z^{-1/4}, z^{1/4}) \begin{pmatrix}
1 & 0 & 0  \\
0& 1 &-i   \\
0 & 1  & i
\end{pmatrix}, \quad z\to \infty , \quad z\in \C\setminus \R .
\end{align*}

In order to solve this RH problem we use the polynomial $D(\zeta)$
\begin{equation}
\label{defD}
D(\zeta) = \zeta (\zeta-q_+)(\zeta-q_-)=\zeta (\zeta^2+\zeta-1)
\end{equation}
(see \eqref{qPlusMinus}). The square root $D (\zeta)^{1/2}$, which branches at $0$ and $q_\pm$, is defined with a cut on
$\zeta_{2-} (-\infty, p_-)\cup\zeta_{2-}(0,p_+)$, which, as noted before,
are the parts of the boundary of $\widetilde{\mathcal R}_2$ that are
in the upper half of the $\zeta$-plane.
We assume that $D (\zeta)^{1/2}>0 $ for $\zeta>q_+$, so that $D (\zeta)^{1/2}<0$ in $(q_-, 0)$ and $D (\zeta)^{1/2}\in i \R_-$ for $z\in ( 0, q_+)$.

We construct the matrix $\widehat{\bm N}(z)$ as follows
\begin{equation} \label{constructionN}
\widehat{\bm N}(z) =
 \begin{pmatrix}
F_1(\zeta_1(z)) & F_1(\zeta_2(z)) & F_1(\zeta_3(z)) \\
F_2(\zeta_1(z)) & F_2(\zeta_2(z)) & F_2(\zeta_3(z)) \\
F_3(\zeta_1(z)) & F_3(\zeta_2(z)) & F_3(\zeta_3(z))
\end{pmatrix},
\end{equation}
where
\[ F_1(\zeta) = K_1\frac{\zeta^2}{D(\zeta)^{1/2}}, \qquad F_2(\zeta) = K_2\frac{\zeta(\zeta-1)}{D(\zeta)^{1/2}}, \qquad F_3(\zeta) = K_3\frac{(\zeta-1)(\zeta - \zeta^*)}{D(\zeta)^{1/2}},
\]
with $D(\zeta)$ given by \eqref{defD}, and $K_1, K_2, K_3, \zeta^*$ are constants to be computed.

Because of the branch cut for $D(\zeta)^{1/2}$ the functions $F_j$, $j=1,2,3,$ defined above satisfy
$$
F_{j+}(\zeta) = - F_{j-}(\zeta) , \quad \zeta\in \partial \widetilde{\mathcal R}_2 \cap (\Im \zeta>0), \quad \text{and} \quad F_{j+}(\zeta) =  F_{j-}(\zeta) , \quad \zeta\in \partial \widetilde{\mathcal R}_2 \cap (\Im \zeta<0).
$$
Consequently, on $(0, p_+)$, for $j=1,2,3$ we have
\begin{align*}
\widehat{\bm N}_{j1+}(z) & =F_j(\zeta_{1+}(z)) = - F_j(\zeta_{2-}(z)) = - \widehat{\bm N}_{j2-}(z), \\
\widehat{\bm N}_{j2+}(z) & =F_j(\zeta_{2+}(z)) =  F_j(\zeta_{1-}(z)) =  \widehat{\bm N}_{j1-}(z), \\
\widehat{\bm N}_{j3+}(z) & =F_j(\zeta_{3+}(z)) =  F_j(\zeta_{3-}(z)) =  \widehat{\bm N}_{j3-}(z),
\end{align*}
and on $(-\infty, p_-)$,
\begin{align*}
\widehat{\bm N}_{j1+}(z) & =F_j(\zeta_{1+}(z)) =  F_j(\zeta_{1-}(z)) =  \widehat{\bm N}_{j1-}(z), \\
\widehat{\bm N}_{j2+}(z) & =F_j(\zeta_{2+}(z)) =  F_j(\zeta_{3-}(z)) =  \widehat{\bm N}_{j3-}(z), \\
\widehat{\bm N}_{j3+}(z) & =F_j(\zeta_{3+}(z)) = - F_j(\zeta_{2-}(z)) = - \widehat{\bm N}_{j2-}(z).
\end{align*}

In other words,  $\widehat{\bm N}$ constructed by formula \eqref{constructionN} satisfies the jump conditions \eqref{Ntilde1}--\eqref{Ntilde2}. It remains to analyze the asymptotic behavior at infinity.
Notice that, by \eqref{defD}, $D(1)=1$, so that
\[ F_1(\zeta) =  K_1   + \mathcal{O}(\zeta -1)), \quad \zeta \to 1, \quad \mbox{and} \quad F_1(\zeta) = \mathcal{O}(\zeta^{3/2}), \quad \zeta \to 0.
\]
Taking $K_1 = 1$ and using \eqref{zeta1}--\eqref{zeta3} it follows that
\[ \widehat{\bm N}_{11}(z) = 1 + \mathcal{O}(z^{-1}), \quad z \to \infty, \quad \mbox{and} \quad \widehat{\bm N}_{1k}(z) = \mathcal{O} ({z^{-3/4}}),k=2,3, \quad z \to \infty.
\]
We also have that
\[ F_j(\zeta) = \mathcal{O}(\zeta -1), \quad \zeta \to 1, \quad j=2,3;
\]
consequently, from \eqref{zeta1}
\[ \widehat{\bm N}_{j1}(z) = \mathcal{O}(z^{-1}), \quad z \to \infty, \quad j=2,3.
\]

With our convention about the branch of $D(\zeta)^{1/2}$, we see that
$$
D(\zeta)^{1/2} = -i \zeta^{1/2}+ \mathcal{O}({\zeta^{3/2}}), \qquad \zeta \to 0,
$$
where the (main) branch cut of $\zeta^{1/2}$ goes along the arc $\zeta_{2-} (-\infty, p_-)$, which joins $q_-$ and $0$ in the upper half plane, and the ray $(-\infty, q_-]$.
Thus,
\[ F_2(\zeta) =  - i K_2  \,  \zeta^{1/2} + \mathcal{O}({\zeta^{3/2}}),  \qquad \zeta \to 0.
\]

Using \eqref{zeta2} and \eqref{zeta3}, it follows that
\[ \widehat{\bm N}_{22}(z) = -i K_2\, z^{-1/4}  + \mathcal{O}({z^{-3/4}}), \qquad z \to \infty.
\]
\[ \widehat{\bm N}_{23}(z) = - K_2\,  z^{-1/4}   + \mathcal{O}({z^{-3/4}}), \qquad z \to \infty.
\]
Choosing
$$
K_2=   2^{-3/2}i,
$$
we find that, as $z \to \infty$,
\[ \widehat{\bm N}_{22}(z) =   2^{-3/2} z^{-1/4}\left(1 +  \mathcal{O}({z^{-1/2}}) \right), \quad
\widehat{\bm N}_{23}(z) =  - 2^{-3/2} i z^{-1/4} \left(1 +  \mathcal{O}({z^{-1/2}}) \right), 
\]
as needed.

Analogously,
\[ F_3(\zeta) =  K_3 \left( i \zeta^*  \,  \zeta^{-1/2} - \frac{i(2+\zeta^*)}{2} \,  \zeta^{1/2} + \mathcal{O}({\zeta^{3/2}})\right), \qquad \zeta \to 0,
\]
and substituting $\zeta_2(z)$ and $\zeta_3(z)$ into $F_3$, we find that
\[ \widehat{\bm N}_{32}(z) = K_3 \left( i\zeta^* z^{1/4} -i(1+\zeta^*) z^{-1/4} + \mathcal{O}({z^{-3/4}})\right), \qquad z \to \infty.
\]
\[ \widehat{\bm N}_{33}(z) =  K_3 \left( - \zeta^* z^{1/4} - (1+\zeta^*) z^{-1/4} + \mathcal{O}({z^{-3/4}})\right), \qquad z \to \infty.
\]
Hence, choosing $\zeta^*=-1$ and   $K_3=  K_2=2^{-3/2}i $, we obtain
\[ \widehat{\bm N}_{32}(z) =  2^{-3/2} z^{1/4} \left(1 +  \mathcal{O}({z^{-1/2 }}) \right), \quad
\widehat{\bm N}_{33}(z) =    2^{-3/2} i z^{1/4} \left(1 +  \mathcal{O}({z^{-1/2 }}) \right), \quad z \to \infty.
\]

The matrix $\widehat{\bm N}(z)$ has the following behavior near the (finite) branch points
\begin{equation} \label{Np-}
\widehat{\bm N}(z) = \left(
\begin{array}{ccc}
1 & |z - p_-|^{-1/4} & |z - p_-|^{-1/4} \\
1 & |z - p_-|^{-1/4} & |z - p_-|^{-1/4} \\
1 & |z - p_-|^{-1/4} & |z - p_-|^{-1/4}
\end{array}
\right), \qquad z \to p_-,
\end{equation}
\begin{equation} \label{Np+}
\widehat{\bm N}(z) = \left(
\begin{array}{ccc}
|z - p_+|^{-1/4} & |z - p_+|^{-1/4} & 1\\
|z - p_+|^{-1/4} & |z - p_+|^{-1/4} & 1\\
|z - p_+|^{-1/4} & |z - p_+|^{-1/4} & 1
\end{array}
\right), \qquad z \to p_+,
\end{equation}
and
\begin{equation} \label{N0}
\widehat{\bm N}(z) = \left(
\begin{array}{ccc}
|z |^{-3/4} & |z |^{-1/4} & 1\\
|z |^{-3/4} & |z |^{-1/4} & 1\\
|z |^{-3/4} & |z |^{-1/4} & 1
\end{array}
\right), \qquad z \to 0.
\end{equation}

Indeed, for $j=1,2,3,$ we have that $F_j(\zeta) = \mathcal{O}((\zeta - q_-)^{-1/2})$, $\zeta \to q_-$. Now, $\zeta_2^{-1}(q_-) = \zeta_3^{-1}(q_-) = p_-$ is a first order finite branch point. On the other hand, $p_-$ is a regular point of $\zeta_1$ and the image by $\zeta_1$ of a sufficiently small neighborhood of $p_-$ remains bounded away from all the singularities of $F_j$, $j=1,2,3$. This gives \eqref{Np-}. Analogously, for $j=1,2,3,$ we have that $F_j(\zeta) = \mathcal{O}((\zeta - q_+)^{-1/2})$, $\zeta \to q_+$. Now, $\zeta_1^{-1}(q_+) = \zeta_2^{-1}(q_+) = p_+$ is a first order finite branch point. We also have that $p_+$ is a regular point of $\zeta_3$ and the image by $\zeta_3$ of a sufficiently small neighborhood of $p_+$ remains bounded away from all the singularities of $F_j$, $j=1,2,3$,  and \eqref{Np+} follows. Finally, $F_1(\zeta) = \mathcal{O}(\zeta^{3/2})$, $F_2(\zeta) = \mathcal{O}(\zeta^{1/2})$, $F_3(\zeta) = \mathcal{O}(\zeta^{-1/2})$, as $\zeta \to \infty$. Since $\zeta_1^{-1}(
 \infty) = \zeta_2^{-1}(\infty) = 0$ is a first order finite branch point, while $0$ is a regular point of $\zeta_3$, whose image of a small neighborhood of $0$ is away from the singularities of $F_j$, $j=1,2,3$, we obtain \eqref{N0}.

Notice that \eqref{Ntilde1}--\eqref{Ntilde2} imply that $\det \widehat{\bm N}(z)$ is analytic in $\mathbb{C} \setminus \{0,p_-,p_+\}$. This together with \eqref{Np-}--\eqref{N0} gives us that $\det \widehat{\bm N}(z)$ is a entire function. From the asymptotic behavior of  $\widehat{\bm N}(z)$ at $\infty$, it follows that $\lim_{z\to \infty} \det \widehat{\bm N}(z) = -i/2$; therefore
\[ \det \widehat{\bm N}(z) \equiv -i/2, \qquad z \in \mathbb{C}.
\]

We can summarize our findings as follows:
\begin{proposition}\label{prop:global}
A solution of the model Riemann-Hilbert problem \eqref{RH-N1}--\eqref{RH-N3} is given by
$$
\bm N(z)=\begin{pmatrix}
F_1(\zeta_1(z)) & F_1(\zeta_2(z)) & F_1(\zeta_3(z)) \\
F_2(\zeta_1(z)) & F_2(\zeta_2(z)) & F_2(\zeta_3(z)) \\
F_3(\zeta_1(z)) & F_3(\zeta_2(z)) & F_3(\zeta_3(z))
\end{pmatrix}\, \diag(1/ f_1(z), 1/  f_2(z), 1/ f_3(z)) ,
$$
where $f_j(z)=r_j(\zeta_j(z))$, $j=1, 2, 3$, functions $r_j$ are given in \eqref{def_ff12}--\eqref{def_ff3}, and
\[
F_1(\zeta) =  \frac{\zeta^{3/2}}{(\zeta^2+\zeta-1)^{1/2}}, \quad F_2(\zeta) =  \frac{2^{-3/2}i \, \zeta^{1/2}(\zeta-1)}{(\zeta^2+\zeta-1)^{1/2}}, \quad F_3(\zeta) = \frac{2^{-3/2}i \, (\zeta^2-1)}{\zeta^{1/2}(\zeta^2+\zeta-1)^{1/2}},
\]
with the branches chosen as specified above.
\end{proposition}

\subsection{Parametrices near  soft edges }

\subsubsection{Parametrix near $p_+$}

In the terminology of   \cite{MR2283089}, we deal here with a (soft) band/void edge.

Consider a small fixed disk, $B_\delta$, of radius $0<\delta<p_+/2$, and center at $p_+$ (see Figure~\ref{fig:localanalysisPplus}). We look for $\bm P$   holomorphic in $ B_\delta \setminus (\R\cup \Gamma^\pm  )$, such that $ \bm P_+(z)=\bm P_-(z) \bm J_{\bm T}(z)$, where, as we have seen,
\begin{align*}
\bm J_{\bm T}(z)& = \begin{pmatrix}
 0  &  \frac{4}{1-\upsilon^{-4n}  } & 0  \\
- \frac{1-\upsilon^{-4n}  }{4} &  0  & 0\\
0& 0& 1
\end{pmatrix},	 \quad x\in ( p_+-\delta, p_+),
\\
\bm J_{\bm T}(z)& =
\begin{pmatrix}
1  &  \frac{4}{(1-\upsilon^{-4n})  \upsilon^{n}}
  e^{n(2 g_{1 }   - g_2 + \omega)(x)}& 0 \\
& 1  & \\
& & 1
\end{pmatrix}, \quad  (p_+, p_++\delta),
\\
\bm J_{\bm T}(z)& = \begin{pmatrix}
1   & 0&0 \\
 \frac{1-\upsilon^{-4n}  }{4} e^{n \psi(z)}  & 1 & 0 \\
& 0& 1
\end{pmatrix} ,	 \quad z\in B_\delta\cap \Gamma^\pm,
\end{align*}
and $\bm P$ is bounded as $z\to p_+$, $z\in   \R \setminus \Gamma^\pm$.

\begin{figure}[t]
\centering \begin{overpic}[scale=1.2]%
{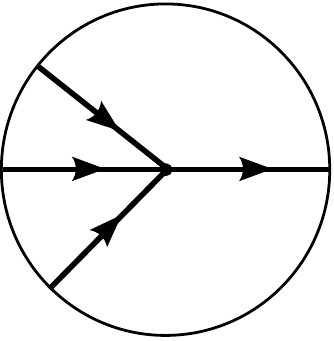}%
      \put(50, 44){$p_+ $}
           \put(100, 44){$p_+  +\delta$}
                    \put(-28, 44){$p_+  -\delta$}
       \put(22, 75){$\Gamma^+ $}
       \put(23,15){$\Gamma^- $}
       \put(88,15){\small $\partial B_\delta $}
\end{overpic}
\caption{Local analysis at $p_+$.}
\label{fig:localanalysisPplus}
\end{figure}

Additionally, as $n\to \infty$, we need
$$
\bm P(z)=\bm N(z) \left( \bm I + \mathcal O(1/n) \right) \quad z\in \partial B_\delta \setminus (\R \cup \Gamma^\pm),
$$
where $\bm N$ is the matrix-valued function described in Proposition~\ref{prop:global}.

We follow the well-known scheme, and build $\bm P$ in the form
\begin{equation} \label{parametrixP}	
 \bm P(z) = \bm E(z)
 \bm    \Psi\left(n^{2/3}f (z)\right)\diag\left( \frac{2}{\left(1-\upsilon^{-4n}\right)^{1/2}} e^{-\frac{n}{2}\psi(z)},
    \frac{\left(1-\upsilon^{-4n}\right)^{1/2}}{2}    e^{\frac{n}{2}\psi(z)},1\right),
\end{equation}
where
\begin{equation} \label{eq:defE}
  \bm   E(z) =   \bm N(z)\,    \begin{pmatrix}
   \sqrt{\pi} & -\sqrt{\pi} & 0 \\
   -i \sqrt{\pi} & -i \sqrt{\pi} & 0 \\
   0 & 0 & 1
   \end{pmatrix} \,
    \begin{pmatrix}
        n^{1/6}f ^{1/4}(z) &   0 & 0 \\
        0 &   n^{-1/6}f ^{-1/4}(z) & 0 \\
        0 &   0 & 1
    \end{pmatrix},
\end{equation}
and
\begin{align} \label{fqdef}
    f(z) =
    \left[\frac{3}{4}\psi(z)\right]^{2/3}
\end{align}
is a biholomorphic (conformal) map of a neighborhood of $p_+$ onto a
neighborhood of the origin such that $f(z)$ is real and positive
for $z>p_+$; recall that $\psi$ was defined in \eqref{def:mappingPsi2}. We may deform the contours $\Gamma^\pm$ near $p_+$ in such a way that $f$ maps
$\Gamma^\pm \cap B_\delta $ to the rays with angles
$\frac{2\pi}{3}$ and $-\frac{2\pi}{3}$, respectively.

Matrix $\bm \Psi$ is build using the Airy functions, as described for instance in \cite[page 253]{MR2470930}.
We put
\begin{equation*}
    y_0(s)=\Ai(s), \quad  y_1(s) = \omega \Ai(\omega s),
    \quad y_2(s) = \omega^{ 2 } \Ai(\omega^{ 2 } s), \quad \omega=e^{2\pi i/3}\,,
\end{equation*}
where $\Ai$ is the usual Airy function. Define the $2\times 2$ matrix  $\bm   K$ by
\begin{align*}
\bm     K(s) & =
    \begin{pmatrix}
        y_0(s) &  -y_2(s)   \\
        y_0'(s) &  -y_2'(s)     \end{pmatrix},
    \quad \arg s \in (0, 2 \pi/3),
    \\
 \bm    K(s) & =
    \begin{pmatrix}
        -y_1(s) &  -y_2(s)   \\
        -y_1'(s) &  -y_2'(s)     \end{pmatrix},
     \quad  \arg s \in (2\pi/3,\pi),
     \\
 \bm    K(s) & =
    \begin{pmatrix}
        -y_2(s) &   y_1(s)   \\
        -y_2'(s) &  y_1'(s)
            \end{pmatrix},
    \quad   \arg s \in (-\pi,-2\pi/3),
    \\
  \bm   K(s) & =
    \begin{pmatrix}
        y_0(s) & y_1(s)   \\
        y_0'(s) & y_1'(s)
            \end{pmatrix},
    \quad   \arg s \in (-2\pi/3,0).
\end{align*}
Then we take the $3\times 3$ matrix $\bm   \Psi$ as
$$
 \bm   \Psi(s) = \left(\begin{MAT}{c.c}
  \bm K (s) & \bm 0 \\.
 \bm  0 & 1 \\
\end{MAT} \right).
$$
This construction uses also identity \eqref{identityPsi1}.

\subsubsection{Parametrix near $p_-$}

In the terminology of   \cite{MR2283089}, this is a (soft) band/saturated region edge.

Consider a small fixed disk, $B_\delta$, of radius $0<\delta<|p_-|/2$, and center at $p_-$ (see Figure~\ref{fig:localanalysisPminus}). We look for $\bm P$   holomorphic in $ B_\delta \setminus (\R\cup \Delta^\pm  )$, such that $ \bm P_+(z)=\bm P_-(z) \bm J_{\bm T}(z)$, where, as we have seen,
\begin{align*}
\bm J_{\bm T}(z)& = \begin{pmatrix}
1 &  &   \\
&1 &   \\
 & 2 x_+^{1/2} e^{-n(g_{2+}+g_{2-}-g_1)}  & 1
\end{pmatrix} ,	 \quad   x\in ( p_- , p_- +\delta),
\\
\bm J_{\bm T}(z)& = \begin{pmatrix}
1 &  &   \\
& 0 & - 1/(2x_+^{1/2}  )  \\
 &  2x_+^{1/2}  & 0
\end{pmatrix} ,	 \quad    (p_- -\delta, p_-),
\\
\bm J_{\bm T}(z)& = \begin{pmatrix}
1 &  &   \\
& 1 &\pm \frac{1}{2  z^{1/2} \upsilon^{2n}}e^{n(2 g_{2 } -g_1)}  \\
& & 1
\end{pmatrix},	  \quad z\in B_\delta\cap \Delta^\pm,
\end{align*}
and $\bm P$ is bounded as $z\to p_-$, $z\in   \R \setminus \Delta^\pm$.

\begin{figure}[t]
\centering \begin{overpic}[scale=1.2]%
{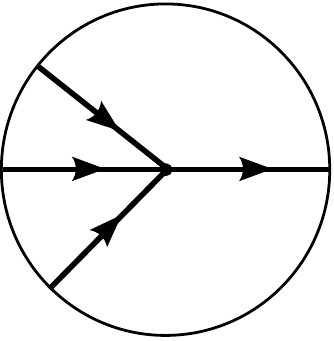}%
      \put(50, 44){$p_- $}
           \put(100, 44){$p_-  +\delta$}
                    \put(-28, 44){$p_-  -\delta$}
       \put(22, 75){$\Delta^+ $}
       \put(23,15){$\Delta^- $}
       \put(88,15){\small $\partial B_\delta $}
\end{overpic}
\caption{Local analysis at $p_-$.}
\label{fig:localanalysisPminus}
\end{figure}

Additionally, as $n\to \infty$, we need
$$
\bm P(z)=\bm N(z) \left( \bm I + \mathcal O(1/n) \right) \quad z\in \partial B_\delta \setminus (\R \cup \Delta^\pm),
$$
where $\bm N$ is the matrix-valued function described in Proposition~\ref{prop:global}.

Let us define
$$
h(z)= \frac{e^{n(2g_2-g_1)(z)/2} }{ 2^{ 1/2} z^{ 1/4}\upsilon^{ n }}, \quad z \in   \C \setminus (-\infty, p_+],
$$
and
\begin{equation} \label{defMmatrix}	
\bm H(z)= \begin{cases}  \diag\left(1,   h(z), 1/ h(z) \right), & z \in B_\delta \cap \{\Im z < 0 \}, \\
\diag\left(1,  i h(z), -i / h(z) \right), & z \in B_\delta \cap \{\Im z > 0 \}.
\end{cases}
\end{equation}
Let also
$$
 \widetilde{\bm P}(z) = \bm P(z) \bm H(z), \quad z \in B_\delta \setminus \R.
$$
Then $ \widetilde{\bm P}_+(z)=\widetilde{\bm P}_-(z) \bm J_{\widetilde{\bm P}}(z)$, with
\begin{align*}
\bm J_{\widetilde{\bm P}}(z)& = \begin{pmatrix}
1 &  &   \\
&1  &   \\
 & -1  & 1
\end{pmatrix} ,	 \quad   x\in ( p_- , p_- +\delta),
\\
\bm J_{\widetilde{\bm P}}(z)& = \begin{pmatrix}
1 &  &   \\
& 0 & 1  \\
 & -1 & 0
\end{pmatrix} ,	 \quad    (p_- -\delta, p_-),
\\
\bm J_{\widetilde{\bm P}}(z)& = \begin{pmatrix}
1 &  &   \\
& 1 & - 1 \\
& & 1
\end{pmatrix},	  \quad z\in B_\delta\cap \Delta^\pm,
\end{align*}
and $\widetilde{\bm P}(z) = \mathcal O (1, z^{-1/4}, z^{1/4})$ as $z\to p_-$, $z\in   \R \setminus \Delta^\pm$.

Comparing it with the RH problem for $\bm    K$ above (see e.g.~\cite[p.~213]{MR2000g:47048}) we see that non-trivial jumps for $\widetilde{\bm P} $ coincide with those of $\sigma_1 \sigma_3 \bm K(z) \sigma_1 \sigma_3$, where
$$
  \sigma_1 \sigma_3 = \begin{pmatrix} 0 & -1 \\ 1 & 0
\end{pmatrix}
$$
(see also \cite[formula (5.22)]{MR2283089}).

Then, taking
$$
\bm \Psi(s) = \left(\begin{MAT}{c.c}
 1 & \bm 0 \\.
 \bm  0 &  \sigma_1 \sigma_3  \bm K (s) \sigma_1 \sigma_3  \\
\end{MAT} \right) ,
$$
as before, we conclude that
\begin{equation} \label{parametrixPhat}	
 \bm P(z) = \bm E(z)
 \bm    \Psi\left(n^{2/3}\widehat f (z)\right)   \bm H^{-1}(z),
\end{equation}
where
\begin{equation} \label{eq:defEhatPminus}
  \bm   E(z) =   \bm N(z)\,    \begin{pmatrix}
  1 & 0 & 0 \\
 0 &  \sqrt{\pi} & -\sqrt{\pi}   \\
0 &   -i \sqrt{\pi} & -i \sqrt{\pi}
 \end{pmatrix} \,
    \begin{pmatrix}
   1 & 0 & 0 \\
0 &        n^{1/6}\widehat  f ^{1/4}(z) &   0  \\
 0 &       0 &   n^{-1/6}\widehat  f ^{-1/4}(z)
    \end{pmatrix},
\end{equation}
and
\begin{align} \label{fqdefhatPminus}
  \widehat  f(z) =
    \left[\frac{3}{4}\widehat \psi(z)\right]^{2/3}
\end{align}
with $\widehat \psi$ defined in \eqref{def:mappingPsi1}, such that $\widehat f$ is a biholomorphic (conformal) map of a neighborhood of $p_-$ onto a
neighborhood of the origin such that $f(z)$ is real and positive
for $z>p_-$.

\subsubsection{Parametrix near the origin (hard edge) }

\begin{figure}[t]
\centering \begin{overpic}[scale=1.2]%
{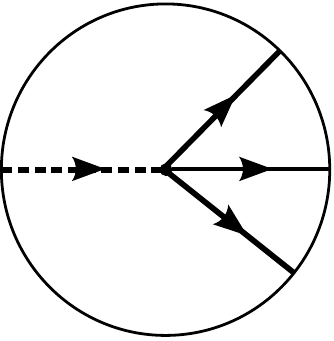}%
      \put(45, 41){$0 $}
           \put(100, 44){$ \delta$}
                    \put(-15, 44){$   -\delta$}
       \put(62, 76){$\Gamma^+ $}
       \put(63,20){$\Gamma^- $}
       \put(-6,15){\small $\partial B_\delta $}
\end{overpic}
\caption{Local analysis at the origin.}
\label{fig:localanalysis0}
\end{figure}
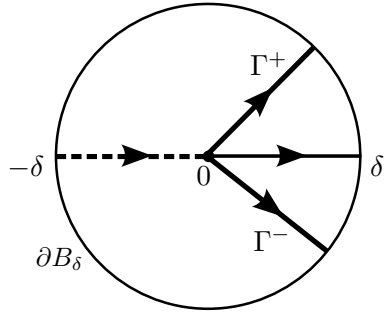

Following the ideas of \cite{MR2470930} (see Section 8.2.1 therein), we consider a small fixed disk, $B_\delta$, of radius $0<\delta<  |p_-|/2  $, centered at the origin (see Figure~\ref{fig:localanalysis0}).
We look for $\bm P$   holomorphic in $ B_\delta \setminus (\R_+\cup \Gamma^\pm  )$, such that $ \bm P_+(z)=\bm P_-(z) \bm J_{\bm T}(z)$, where, as we have seen,
\begin{align*}
\bm J_{\bm T}(z)& = \begin{pmatrix}
 0  &  \frac{4}{1-\upsilon^{-4n}  } & 0  \\
- \frac{1-\upsilon^{-4n}  }{4} &  0  & 0\\
0& 0& 1
\end{pmatrix},	 \quad x\in B_\delta \cap (0,p_+),
\\
\bm J_{\bm T}(z)& = \begin{pmatrix}
1   & 0&0 \\
 \frac{1-\upsilon^{-4n}  }{4} e^{n \psi(z)}  & 1 & 0 \\
& 0& 1
\end{pmatrix} ,	 \quad z\in B_\delta \cap  \Gamma^\pm,
\end{align*}
 such that
\begin{equation} \label{localPat0}	
\bm P(z)=\mathcal O \left(\begin{array}{c|c|c}1 & |z|^{-1/2 } &  |z|^{-1/2 }   \end{array}\right)
  \quad z\to 0,
\end{equation}
and
$$
\bm P(z)=\bm N(z) \left( \bm I + \mathcal O(1/n) \right) \quad z\in \partial B_\delta \setminus (\R_+ \cup \Gamma^\pm).
$$

Observe that at this stage we have disregarded the jump of $\bm T$ on $(-\delta, 0)$, given by
$$
 \begin{pmatrix}
1 &  &   \\
&1 &   \\
 & 2 x_+^{1/2} e^{-n(g_{2+}+g_{2-}-g_1)}  & 1
\end{pmatrix},
$$
because, according to item $(ii)$ of Proposition~\ref{prop:summaryG}, the off-diagonal term converges to $0$ uniformly in $n$.

Let
$$
 \widetilde{\bm P}(z) = \bm P(z) \diag \left(  \frac{2}{\sqrt{1-\upsilon^{-4n} }  } e^{-n \psi(z)/2}, \frac{\sqrt{1-\upsilon^{-4n} }  } {2} e^{n \psi(z)/2}, 1 \right), \quad z \in B_\delta \setminus \R,
$$
with the square root (well defined for $n$ large enough) positive on $\R_+$. Using \eqref{bdryvalues} and \eqref{psiOnNegative} we conclude that
 $ \widetilde{\bm P}$  is also holomorphic in $ B_\delta \setminus (\R_+\cup \Gamma^\pm  )$, with $  \widetilde{\bm P}_+(z)= \widetilde{\bm P}_-(z) \bm J_{ \widetilde{\bm P}}(z)$, where
\begin{align*}
\bm J_{ \widetilde{\bm P}}(z)& = \begin{pmatrix}
 0  &  1 & 0  \\
- 1 &  0  & 0\\
0& 0& 1
\end{pmatrix},	 \quad x\in B_\delta \cap (0,p_+),
\\
\bm J_{ \widetilde{\bm P}}(z)& = \begin{pmatrix}
1   & 0&0 \\
 1 & 1 & 0 \\
& 0& 1
\end{pmatrix} ,	 \quad z\in B_\delta \cap  \Gamma^\pm.
\end{align*}
Also, the local behavior of $\widetilde{\bm P}$ at the origin matches that of $\bm P$  (see \eqref{localPat0}).

Parametrix $\bm P$ will be built in terms of the modified Bessel functions of order $0 $
see \cite[Section 8.2.1]{MR2470930}.
Namely, with the modified Bessel functions $I_{0}$ and $K_{0}$, and the
Hankel functions $H_{0}^{(1)}$ and $H_{0}^{(2)}$ (see \cite[Chapter 9]{abramowitz/stegun:1972}), we define a
$2 \times 2 $ matrix $\bm L(\zeta)$ for $|\arg \zeta| < 2 \pi/3$ as
\begin{equation}\label{RHPPSIsolution1}
    \bm L(\zeta) =
    \begin{pmatrix}
        I_{0 } (2 \zeta^{1/2}) & \frac{i}{\pi} K_{0 }(2 \zeta^{1/2}) \\[1ex]
        2\pi i \zeta^{1/2} I_{0 }'(2\zeta^{1/2}) & -2 \zeta^{1/2} K_{0 }'(2\zeta^{1/2})
    \end{pmatrix}.
\end{equation}
For $2\pi/3 < \arg \zeta < \pi$ we define it as
\begin{equation}\label{RHPPSIsolution2}
    \bm L(\zeta) =
    \begin{pmatrix}
        \frac{1}{2} H_{0 }^{(1)}(2(-\zeta)^{1/2}) &
        \frac{1}{2} H_{0 }^{(2)}(2(-\zeta)^{1/2}) \\[1ex]
        \pi \zeta^{1/2} \left(H_{0 }^{(1)}\right)'(2(-\zeta)^{1/2}) &
        \pi \zeta^{1/2} \left(H_{0 }^{(2)}\right)'(2(-\zeta)^{1/2})
    \end{pmatrix}  .
\end{equation}
And finally for $- \pi < \arg \zeta  < -2\pi/3$ it is defined as
\begin{equation}\label{RHPPSIsolution3}
    \bm L(\zeta) =
    \begin{pmatrix}
        \frac{1}{2} H_{0 }^{(2)}(2(-\zeta)^{1/2}) &
        -\frac{1}{2} H_{0 }^{(1)}(2 (-\zeta)^{1/2}) \\[1ex]
        -\pi \zeta^{1/2} \left(H_{0 }^{(2)}\right)'(2 (- \zeta)^{1/2}) &
        \pi \zeta^{1/2} \left(H_{0 }^{(1)}\right)'(2 (-\zeta)^{1/2})
    \end{pmatrix}.
\end{equation}
With this definition we take
$$
\bm \Psi(s) = \left(\begin{MAT}{c.c}
\sigma_3  \bm L (-s) \sigma_3 & \bm 0 \\.
 \bm  0 & 1 \\
\end{MAT} \right) , \quad \sigma_3= \begin{pmatrix}
1 & 0   \\
0 &   -1
\end{pmatrix}.
$$
As in \cite{MR2470930}, we conclude that
  \begin{equation} \label{parametrixPhatat0}	
 \bm P(z) = \bm E(z)
 \bm    \Psi\left(n^2   f (z)\right)  \diag \left(  \frac{\sqrt{1-\upsilon^{-4n} }  } {2} e^{n \psi(z)/2}, \frac{2}{\sqrt{1-\upsilon^{-4n} }  } e^{-n \psi(z)/2},  1 \right),
\end{equation}
where
\begin{equation} \label{eq:defEhat}
  \bm   E(z) =   \bm N(z)\,    \diag \left(\frac{1}{\sqrt{2}} \begin{pmatrix} 1 & i \\ i & 1 \end{pmatrix}, 1 \right)
       \diag \left( (2\pi n)^{1/2} f(z)^{1/4},\; (2\pi n)^{-1/2} f(z)^{-1/4}, \; 1
    \right),
\end{equation}
and
\begin{align} \label{fqdefhat}
     f(z) =
    \left[\frac{3}{4} ( \psi(z)-  \psi(0))\right]^{2/3}.
\end{align}

\medskip

 \subsection{Final transformation}\label{section9}

Recall that we denote generically by $B_{\delta}$ the small disks around the branch points $0$  and $p_\pm$,
and by $\bm P$ the local parametrices built in $B_{\delta}$. We define the matrix valued function $\bm R$ as
\begin{equation} \label{Rdef}
\bm R(z) = \begin{cases}
    \bm T(z) \bm P^{-1}(z), & \text{in the  neighborhoods $B_{\delta}$,   }   \\
    \bm T(z) \bm N^{-1}(z), & \text{elsewhere.}
    \end{cases}
\end{equation}
Then $\bm R$ is defined and analytic outside the real line, the lips $\Delta^\pm $
and $\Gamma^\pm $ of the lenses and the circles around the three branch points.
The jump matrices of $\bm T$ and $\bm N$ coincide on $(-\infty, p_-)$ and $(0,p_+)$
and the jump matrices of $\bm T$ and $\bm P$ coincide inside the three disks with the
exception of the interval   $(-\delta, 0)$.
It follows that $\bm R$ has an analytic continuation to the complex plane
minus the
contours shown in Figure~\ref{fig:thirdcontour_bis}.

\begin{figure}[t]
\centering \begin{overpic}[scale=1]%
{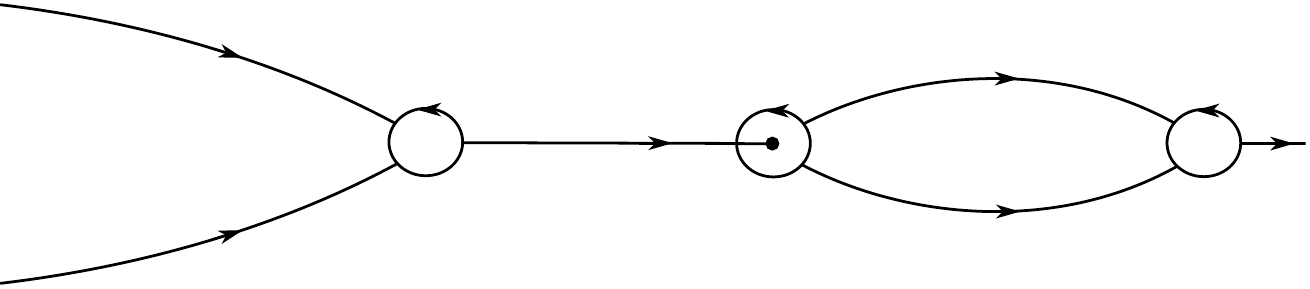}%
      \put(90.5,6.5){$p_+ $}
       \put(20,18){$\Delta^+ $}
       \put(20,2){$\Delta^- $}
       \put(70,16.5){$\Gamma^+ $}
       \put(70,3){$\Gamma^- $}
       \put(31.5,6.5){$p_-$}
\put(59.7,10.3){\scriptsize $0$}
\end{overpic}
\caption{Contours for $\bm R$.}
\label{fig:thirdcontour_bis}
\end{figure}

We can follow the arguments in \cite[Section 9]{MR2470930} to conclude that
\begin{equation} \label{RnearI}
\bm R(z)= \bm I + \mathcal{O} \left( \frac{1}{n (|z|+1)}\right)\,,
    \qquad n \to \infty\,,
\end{equation}
uniformly for $z$ in the complex plane outside of these contours.

\section{Asymptotics }

Now we unravel all the transformations in order to get the asymptotic results from Theorems~\ref{asymptoticsFinal}
and \ref{asymptoticsKernel}.

Assume first that $z$ lies outside  the small disks $B_{\delta}$ around the branch points $0$  and $p_\pm$, so that
$$
\bm T(z) = \bm R(z) \bm N (z) = \left( \bm I + \mathcal{O} \left( \frac{1}{n (|z|+1)}\right) \right)\bm N (z).
$$
Assume further that $z$ lies in one of  the unbounded component of the complement to the curves depicted in Figure~\ref{fig:thirdcontour_bis}. By \eqref{defX1}, \eqref{defUU} and \eqref{defT},
\begin{align*}
 \bm Y(z)  & = \diag\left( e^{-n  \omega },   1 ,  1 \right)  \left( \bm I + \mathcal{O} \left( \frac{1}{n (|z|+1)}\right) \right) \\
 &\times \bm N (z)  \diag\left( e^{n(g_1(z)+\omega)},   e^{-n(g_1(z)-g_2(z))} ,  e^{-n g_2(z)}\right) \left(\begin{MAT}{c.c}
  1 & \bm 0 \\.
 \bm  0 & \bm A_*^{-1}(z) \\
\end{MAT} \right),
\end{align*}
where $A_*$ stands either for $A_L$ or $A_R$.

Thus,
\begin{align*}
 \bm Y_{11}(z)  & =  \left( 1,   0, 0  \right)  \bm Y (z) \begin{pmatrix} 1 \\ 0 \\ 0
 \end{pmatrix} \\
& = \left( e^{-n  \omega },   0 ,  0 \right)  \left( \bm I + \mathcal{O} \left( \frac{1}{n (|z|+1)}\right) \right)
   \bm N (z)   \begin{pmatrix} e^{n(g_1(z)+\omega)} \\ 0 \\ 0
 \end{pmatrix} \\
& = e^{n g_1(z) } \left( 1,   0 ,  0 \right)  \left( \bm I + \mathcal{O} \left( \frac{1}{n (|z|+1)}\right) \right)
   \bm N (z)   \begin{pmatrix} 1 \\ 0 \\ 0
 \end{pmatrix} \\
 & = e^{n g_1(z) } \left( 1,   0 ,  0 \right)  \left( \bm I + \mathcal{O} \left( \frac{1}{n (|z|+1)}\right) \right)
   \bm N_{*1} (z)     \\
& =  e^{n g_1(z) } \left(\bm N_{11}(z)  +  \mathcal{O} \left( \frac{1}{n (|z|+1)}\right) \right).
\end{align*}
It remains to use Proposition~\ref{prop:global} to establish \eqref{outerForQ1}.

In the same fashion, if $z$ lies on the $+$-side of $(0,p_+)$, that is, in a domain of the form
$$
\Omega=\{ z\in \C:\, \Re z \in (\varepsilon, p_+-\varepsilon), \, \Im z \in [0, \varepsilon)\},
$$
where $\varepsilon>0$ is fixed, then
\begin{align*}
 \bm Y(z)  & = \diag\left( e^{-n  \omega },   1 ,  1 \right)  \left( \bm I + \mathcal{O} \left( \frac{1}{n }\right) \right) \bm N (z) \begin{pmatrix}
1   & 0&0 \\
 \frac{1-\upsilon^{-4n}  }{4} e^{n \psi(z)}  & 1 & 0 \\
0 & 0& 1
\end{pmatrix}  \\
 &\times  \diag\left( e^{n(g_1(z)+\omega)},   e^{-n(g_1(z)-g_2(z))} ,  e^{-n g_2(z)}\right) \left(\begin{MAT}{c.c}
  1 & \bm 0 \\.
 \bm  0 & \bm A_R^{-1}(z) \\
\end{MAT} \right),
\end{align*}
so that
\begin{align*}
 \bm Y_{11}(z)  = &   \left( 1,   0, 0  \right)  \bm Y (z) \begin{pmatrix} 1 \\ 0 \\ 0
 \end{pmatrix} \\
 =&  e^{n g_1(z) } \left( 1,   0 ,  0 \right)  \left( \bm I + \mathcal{O} \left( \frac{1}{n }\right) \right)
   \bm N (z)   \begin{pmatrix} 1 \\  \frac{1-\upsilon^{-4n}  }{4} e^{n \psi(z)}  \\ 0
 \end{pmatrix} \\
= & e^{n g_1(z) } \left( 1,   0 ,  0 \right)  \left( \bm I + \mathcal{O} \left( \frac{1}{n }\right) \right)
  \left( \bm N_{*1} (z) + \frac{1-\upsilon^{-4n}  }{4} e^{n \psi(z)} \bm N_{*2} (z) \right)     ,
\end{align*}
which proves \eqref{outerForQ2} with the aid of Proposition~\ref{prop:global}. Now formula \eqref{outerForQ3} follows from \eqref{bdryvalues}.


In the same vein, 
\begin{align*}
 \bm Y_+(x) \begin{pmatrix} 1 \\ 0 \\ 0
 \end{pmatrix}
 & =  e^{n \omega } \diag\left( e^{-n  \omega },   1 ,  1 \right)  \bm R_+(x) \bm N_+ (x) \begin{pmatrix}
e^{n g_{1+}(x)}     \\
 \frac{1-\upsilon^{-4n}(x)  }{4} e^{n g_{1-}(x))}  \\
0  
\end{pmatrix}  ,
\end{align*}
and
\begin{align*}
\begin{pmatrix}
0, w_{1,n}(y), w_{2,n}(y)
\end{pmatrix}
  \bm Y_+(y)^{-1}  & = \frac{4  e^{- n  \omega }}{1-\upsilon^{-4n}(y)}    \begin{pmatrix}
- \frac{1-\upsilon^{-4n} (y) }{4} e^{- n g_{1+}(y)} , e^{- ng_{1-}(y)} , 0
\end{pmatrix} 
  \\
 &\times \bm N_+^{-1} (y) \bm R_+^{-1}(y)  
  \diag\left( e^{n  \omega },   1 ,  1 \right)   ,
\end{align*}
where we have used the explicit expression for $\bm A_{ R}$, the boundary values \eqref{bdryvalues}, and the equilibrium conditions $(i)$ from Proposition~\ref{prop:summaryG}. In consequence, by formula~\eqref{charactKernel},
\begin{align*}
K_n(x,y) & = \frac{1}{2\pi i (x-y)}  \frac{4}{1-\upsilon^{-4n}(y)}  \begin{pmatrix}
- \frac{1-\upsilon^{-4n} (y) }{4} e^{- n g_{1+}(y)} , e^{- ng_{1-}(y)} , 0
\end{pmatrix} 
  \\
 &\times \bm N_+^{-1} (y) \bm R_+^{-1}(y)  
 \bm R_+(x) \bm N_+ (x) \begin{pmatrix}
e^{n g_{1+}(x)}     \\
 \frac{1-\upsilon^{-4n}(x)  }{4} e^{n g_{1-}(x))}  \\
0  
\end{pmatrix}.
\end{align*}
We have
\begin{align*}
\bm N_+^{-1} (y) \bm R_+^{-1}(y)  
 \bm R_+(x) \bm N_+ (x) & = \bm N_+^{-1} (y) \left(\bm I +\mathcal O \left( \frac{x-y}{n}\right) \right) \bm N_+ (x)
 \\
 & = \bm I +\mathcal O \left(  x-y\right)\quad  \text{as} \quad y\to x.
\end{align*}
Thus,
\begin{align*}
K_n(x,y) & = \frac{1}{2\pi i (x-y)}  \frac{4}{1-\upsilon^{-4n}(y)}  \begin{pmatrix}
- \frac{1-\upsilon^{-4n} (y) }{4} e^{- n g_{1+}(y)} , e^{- ng_{1-}(y)} , 0
\end{pmatrix} 
  \\
 &\times \left(    \bm I +\mathcal O \left(  x-y\right) \right) \begin{pmatrix}
e^{n g_{1+}(x)}     \\
 \frac{1-\upsilon^{-4n}(x)  }{4} e^{n g_{1-}(x))}  \\
0  
\end{pmatrix}\\
&= \frac{1}{2\pi i (x-y)}  \left( -e^{- n (g_{1+}(x)-g_{1+}(y))} + \frac{1-\upsilon^{-4n}(x)}{1-\upsilon^{-4n}(y)} e^{- n (g_{1-}(x)-g_{1-}(y))} +\mathcal O(x-y)\right)    \\
&= \frac{1}{2\pi i (x-y)}  \left( -e^{- n (g_{1+}(x)-g_{1+}(y))} +   e^{- n (g_{1-}(x)-g_{1-}(y))}\right) +\mathcal O(1), \quad y\to x.
\end{align*}
Using \eqref{onsupp1} we conclude that
$$
K_n(x,x)=n \lambda_1'(x) + \mathcal O(1), \quad n \to \infty.
$$

On the other hand, if we take
$$
x_n=x^*+\frac{x}{n \lambda_1'(x^*)}, \quad y_n=x^*+\frac{y}{n \lambda_1'(x^*)},
$$
we get
\begin{align*}
K_n(x_n,y_n) &  = \frac{n \lambda_1'(x^*)}{\pi (x-y)}  \left( -e^{- n (g_{1+}(x_n)-g_{1+}(y_n))} +   e^{- n (g_{1-}(x_n)-g_{1-}(y_n))} +\mathcal O\left(\frac{1}{n}\right)\right)\\ &  = \frac{n \lambda_1'(x^*)}{\pi (x-y)}  \left( e^{\pi i (x-y) } - e^{-\pi i (x-y) } +\mathcal O\left(\frac{1}{n}\right)\right).
\end{align*}
This concludes the proof of Theorem~\ref{asymptoticsKernel}.

\section*{Acknowledgements}

The first author (AIA) received support from RFBR grant 13-01-12430 (OFIm) and the Excellence Chair Program sponsored by Universidad Carlos III de Madrid and the Bank of Santander. The second (GLL) and the third (AMF) authors were supported by MICINN of Spain under grants MTM2012-36732-C03-01 and MTM2011-28952-C02-01, respectively, and by the European Regional Development Fund (ERDF). Additionally, AMF was supported by Junta de Andalucía (the Excellence Grant P11-FQM-7276 and the research group FQM-229) and by Campus de Excelencia Internacional del Mar (CEIMAR) of the University of Almería. 

This work was completed during a visit of AMF to the Department of Mathematics of the Vanderbilt University. He acknowledges the hospitality of the hosting department, as well as a partial support of the University of Almer\'{i}a through the travel grant EST2014/046.


\def\cprime{$'$}

\obeylines
\texttt{
A. I. Aptekarev (aptekaa@keldysh.ru)
Keldysh Institute of Applied Mathematics, Moscow, RUSSIA
\medskip
G.~L\'opez Lagomasino (lago@math.uc3m.es)
Department of Mathematics, Universidad Carlos III de Madrid, Legan\'es, SPAIN
\medskip
A. Mart\'{\i}nez-Finkelshtein (andrei@ual.es)
Department of Mathematics
University of Almer\'{\i}a, SPAIN, and
Instituto Carlos I de F\'{\i}sica Te\'{o}rica y Computacional
Granada University, SPAIN
}

\end{document}